\documentclass[production,11pt]{style}\usepackage{latexsym}\usepackage{amsmath}\usepackage{amsthm}\usepackage{xspace}\usepackage{enumerate}\usepackage{amsfonts}\usepackage{amscd}\usepackage{syntonly}\usepackage{amssymb}

\swapnumbers

\newtheorem{thm}{Theorem}
\newtheorem{prop}[thm]{Proposition}
\newtheorem{lemma}[thm]{Lemma}
\newtheorem{claim}{Claim}

\newtheorem{rmk}[thm]{Remark}
\newtheorem{cor}[thm]{Corollary}

\newenvironment{pf}[1][Proof.]{\noindent \emph{#1.}}{}
\newenvironment{enui}{\begin{enumerate}[(i)]}{\end{enumerate}}
\newenvironment{enua}{\begin{enumerate}[(a)]}{\end{enumerate}}

\def\Int{{\operatorname{int}}}
\def\diag{{\operatorname{diag}}}
\def\const{\equiv}
\def\op{{\operatorname{op}}}

\def\Dot{\dot{\phantom{.}}}

\def\Crit{{\operatorname{Crit}}}
\def\codim{{\operatorname{codim}}}
\def\corank{{\operatorname{corank\,}}}
\def\Colon{:} 

\def\Ham{{\operatorname{Ham}}}
\def\c{{\operatorname{c}}}
\def\b{{\operatorname{b}}}
\def\hom{{\operatorname{hom}}}
\def\rank{\operatorname{rank}}

\def\Fix{{\operatorname{Fix}}}

\def\Then{\,\Longrightarrow\,}
\def\im{{\operatorname{im}}}

\def\hol{{\operatorname{hol}}}
\def\F{{\mathcal{F}}}
\def\id{{\operatorname{id}}}
\def\empty{\emptyset}
\def\nn{{\nonumber}}
\newcommand\wt[1]{{\widetilde{#1}}}
\newcommand{\BAR}[1]{{\overline{#1}}}
\def\al{{\alpha}}
\def\be{\beta}

\def\eps{\varepsilon}
\def\Om{\Omega}
\def\om{\omega}
\def\lam{\lambda}
 
\def\si{\sigma}

\def\Th{\Theta}

\renewcommand\phi{\varphi}

\def\SU{\operatorname{SU}}
\def\U{\operatorname{U}}

\newcommand{\N}{\mathbb{N}}
\newcommand{\Z}{\mathbb{Z}}

\newcommand{\R}{\mathbb{R}}

\def\C{\mathbb C}

\def\D{\mathbb{D}}

\def\g{\mathfrak g}
\def\A{\mathcal A}
\def\Lie{\operatorname{Lie}}

\def\pr{{\operatorname{pr}}}
\def\can{{\operatorname{can}}}

\def\sub{\subseteq}

\def\x{\times}
\def\wo{\setminus}
\def\one{\mathbf{1}}
\def\iso{\cong}
\def\d{\partial}

\def\lan{\langle}
\def\ran{\rangle}

\def\Wlog{w.l.o.g.\xspace}

\title[Coisotropic Submanifolds and Leafwise Fixed Points]{Coisotropic Submanifolds, Leafwise Fixed Points, 
and Presymplectic Embeddings} 

\author{Fabian Ziltener (University of Toronto)}

\begin{document}


\begin{abstract} Let $(M,\om)$ be a geometrically bounded symplectic manifold, $N\sub M$ a closed, regular (i.e.  ``fibering'') coisotropic submanifold, and $\phi\Colon M\to M$ a Hamiltonian diffeomorphism. The main result of this article is that the number of leafwise fixed points of $\phi$ is bounded below by the sum of the $\Z_2$-Betti numbers of $N$, provided that the Hofer distance between $\phi$ and the identity is small enough and the pair $(N,\phi)$ is non-degenerate. The bound is optimal if there exists a $\Z_2$-perfect Morse function on $N$. A version of the Arnol'd-Givental conjecture for coisotropic submanifolds is also discussed. As an application, I prove a presymplectic non-embedding result. 
\end{abstract} 

\maketitle
\tableofcontents


\section{Main results}\label{sec:main}
\subsection*{Leafwise fixed points}
Let $(M,\om)$ be a symplectic manifold. We denote by $\Ham(M,\om)$ the group of Hamiltonian diffeomorphisms (see Section \ref{sec:back}). Let $\phi\in\Ham(M,\om)$ and $N\sub M$ be a coisotropic submanifold. We denote by $N_x:=N^\om_x\sub N$ the isotropic leaf through $x$ (see Section \ref{sec:back}). A leafwise fixed point of $\phi$ is by definition a point $x\in N$ such that $\phi(x)\in N_x$. We denote by $\Fix(\phi,N):=\Fix(\phi,N,\om)$ the set of such points. The first main result of this article addresses the following question:\\

\noindent{\bf Question A:} Provided that $\phi$ is close to the identity in a suitable sense, what lower bound on the number $\big|\Fix(\phi,N)\big|$ is there?\\

Note that if $N=M$ then $N_x=\{x\}$, for every $x\in N$, and hence $\Fix(\phi,N)$ is the set $\Fix(\phi)$ of ordinary fixed points of $\phi$. In the other extreme case $\dim N=\dim M/2$ the submanifold $N$ is Lagrangian, and we have $\Fix(\phi,N)=N\cap \phi^{-1}(N)$, provided that $N$ is connected. In order to state the first main result, we denote by $A(M,\om,N)$ the minimal area of $(M,\om,N)$ (see (\ref{eq:A M om N}) below), and by $d:=d^{M,\om}$ the Hofer distance (see (\ref{eq:d phi psi})). We call $N$ regular iff its isotropic leaf relation (see Section \ref{sec:back}) is a closed subset and a submanifold of $N\x N$. Assuming that $N$ is closed, this means that there exists a manifold structure on the set $N_\om$ of isotropic leaves of $N$ such that the canonical projection $\pi_N:N\to N_\om$ is a smooth fiber bundle. For the definitions of (geometric) boundedness of $(M,\om)$ and non-degeneracy for $(N,\phi)$ see Section \ref{sec:back}. The former is a mild condition on $(M,\om)$, examples include closed (compact without boundary) symplectic manifolds, cotangent bundles of closed manifolds, and symplectic vector spaces. Non-degeracy of $(N,\phi)$ naturally generalizes the usual non-degeneracy in the cases $N=M$ and $\dim N=\dim M/2$. For each topological space $X$, commutative ring $R$ and integer $i$ we denote by $b_i(X,R):=\rank_R H_i(X,R)$ the $i$-th Betti number of $X$ with coefficients in $R$. 
\begin{thm}\label{thm:leaf} Let $(M,\om)$ be a (geometrically) bounded symplectic manifold, and $N\sub M$ a closed, regular coisotropic submanifold. Then there exists a constant $C\in(0,\infty]$ such that $C\geq A(M,\om,N)$ and the following holds. If $\phi\in\Ham(M,\om)$ is such that $(N,\phi)$ is non-degenerate and 
\begin{equation}\label{eq:d phi id M}d(\phi,\id)<C,
\end{equation}
then  
\begin{equation}\label{eq:Fix phi}\big|\Fix(\phi,N)\big|\geq \sum_{i=0}^{\dim N}b_i(N,\Z_2).\end{equation} 
\end{thm}
If $\codim N\neq0,1,\dim M/2$ then this theorem appears to be the first result implying that $\big|\Fix(\phi,N)\big|\geq 2$, without assuming that $\phi$ is $C^1$-close to the identity. It generalizes a result for the case $\dim N=\dim M/2$, which is due to Yu.~V.~Chekanov, see the Main Theorem in \cite{Ch}. The bound (\ref{eq:Fix phi}) is sharp, provided that there exists a $\Z_2$-perfect Morse function on $N$, see Theorem \ref{thm:f} below.  
\subsubsection*{Examples}
\label{xpl:action} A large class of examples of regular coisotropic submanifolds is given as follows. Let $(M,\om)$ be a symplectic manifold, and $G$ a compact, connected Lie group with Lie algebra $\g$. We fix a Hamiltonian action of $G$ on $M$, and an (equivariant) moment map $\mu\Colon M\to \g^*$. Assume that $\mu$ is proper and the action of $G$ on $N:=\mu^{-1}(0)\sub M$ is free. Then $N$ is a closed, regular coisotropic submanifold. As a concrete example, let $0<k\leq n$ be integers, and consider $M:=\C^{k\x n}$ with the standard symplectic structure $\om:=\om_0$, and the action of the unitary group $G:=\U(k)$ on $\C^{k\x n}$ by multiplication from the left. A moment map for this action is given by $\mu(\Th):=\frac i2(\one-\Th\Th^*)$, and $N=\mu^{-1}(0)$ is the Stiefel manifold $V(k,n):=\big\{\Th\in\C^{k\x n}\,\big|\,\Th\Th^*=\one\big\}$. The triple $(M,\om,N)$ satisfies the hypotheses in Theorem \ref{thm:leaf}. Furthermore, we have $A\big(\C^{k\x n},\om_0,V(k,n)\big)=\pi$ (see Proposition \ref{prop:Stiefel} below), and $\sum_{i=0}^{\dim V(k,n)}b_i\big(V(k,n),\Z_2\big)=2^k$. (This follows for example from \cite{GHV}, Theorem I on p.~224, and the fact that the integral cohomology of $V(k,n)$ is torsion-free.) Let $\phi\in\Ham(\C^{k\x n},\om_0)$ be such that $\big(V(k,n),\phi)$ is non-degenerate and $d(\phi,\id)<\pi$. Then Theorem \ref{thm:leaf} implies that $\big|\Fix(\phi,V(k,n))\big|\geq2^k$. 

This bound is sharp, since there exists a $\Z_2$-perfect Morse function on $V(k,n)$ (see example (\ref{Stiefel}) after Theorem \ref{thm:f} below). Moreover, the condition $C\geq A\big(\C^{k\x n},\om_0,V(k,n)\big)$ in Theorem \ref{thm:leaf} is also sharp, in the sense that the conclusion of the theorem is wrong if we choose $C>A\big(\C^{k\x n},\om_0,V(k,n)\big)$, see Proposition \ref{prop:Stiefel}. Note that in the case $k=1$ we obtain an improvement of a result by H. Hofer, \cite{Ho}, Proposition 1.4. That result states that $\Fix(\phi,S^{2n-1})\neq\emptyset$, provided that $d_\c(\phi,\id)\leq\pi$. Here $d_\c$ denotes the compactly supported Hofer distance (see (\ref{eq:d c}) below).

\label{V} Another family of examples of regular coisotropic submanifolds arises as follows. Let $(X,\si)$ be a closed symplectic manifold, $\pi\Colon E\to X$ a closed smooth fiber bundle, and $H\sub TE$ a horizontal subbundle. We define $V^*E$ to be the vertical cotangent bundle of $E$. Its fiber over a point $e\in E$ is the space $T_e^*E_{\pi(e)}$. We denote the zero-section of this bundle by $N$. Furthermore, we define a closed two-form on $V^*E$ as follows. We denote by $\pi_X\Colon V^*E\to X$ the canonical projection, by $\om_\can$ the canonical symplectic form on $T^*E$, and by $\pr^H_e\Colon T_eE\to T_eE_{\pi(e)}$ the linear projection along the subspace $H_e\sub T_eE$, for $e\in E$. We define
\[\label{eq:iota H V}\iota_H\Colon V^*E\to T^*E,\quad \iota_H(e,\al):=\big(e,\al\circ\pr^H_e\big),\quad \Om_{\si,H}:=\pi_X^*\si+\iota_H^*\om_\can.\]
Then $\Om_{\si,H}$ is a closed two-form on $V^*E$. Furthermore, by Proposition 3.2 in \cite{Ma}, there exists an open neighborhood $M$ of the zero-section $N\sub V^*E$ on which $\Om_{\si,H}$ is non-degenerate. We fix such an $M$. Then the submanifold $N\sub M$ is regular coisotropic (see Proposition \ref{prop:N} below). Assume now that the base manifold $X$ is symplectically aspherical, i.e. $\int_{S^2}u^*\si=0$, for every $u\in C^\infty(S^2,X)$. Then $A(M,\Om_{\si,H},N)=\infty$ (see again Proposition \ref{prop:N} below). So in this case the only possible constant $C$ as in Theorem \ref{thm:leaf} is $\infty$, and for this constant condition (\ref{eq:d phi id M}) is vacuous. 
\subsubsection*{Idea of proof of Theorem \ref{thm:leaf}}
Assume that the hypotheses of Theorem \ref{thm:leaf} are satisfied. The strategy of the proof is to find a Lagrangian embedding of $N$ into a suitable symplectic manifold, and then apply the Main Theorem in \cite{Ch}. Recall that $N_\om$ denotes the set of isotropic leaves of $N$. Since $N$ is regular, there exists a unique manifold structure on $N_\om$ such that the projection $\pi_N\Colon N\to N_\om$ is a submersion (see Lemma \ref{le:X/R} below). We denote by $\om_N$ the unique symplectic structure on $N_\om$ such that $\pi_N^*\om_N=\om$, and we define 
\begin{equation}\label{eq:wt M wt om}\wt M:=M\x N_\om,\quad \wt\om:=\om\oplus(-\om_N),
\end{equation}
\begin{equation}\label{eq:iota N wt N} \iota_N\Colon N\to \wt M,\quad \iota_N(x):=(x,N_x),\quad \wt N:=\iota_N(N).
\end{equation} 
Then $\iota_N$ is an embedding of $N$ into $\wt M$ that is Lagrangian with respect to the symplectic form $\wt\om$ on $\wt M$ (see Lemma \ref{le:N} below). In order to satisfy the hypotheses of Chekanov's result, the inequality $A(M,\om,N)\leq A(\wt M,\wt\om,\wt N)$ is crucial. It follows from Key Lemma \ref{le:A M om N} below. The idea of its proof is that given a smooth map $\wt u=(v,w')\Colon \D\to \wt M=M\x N_\om$ such that $\wt u(S^1)\sub\wt N$, we may lift $w'$ to a map $w\Colon [0,1]\x S^1\to N$ and concatenate this with $v$. We thus obtain a map $u\Colon \D\to M$ with boundary on an isotropic leaf, satisfying $\int u^*\om=\int\wt u^*\wt\om$. The method described here generalizes a standard way of reducing the case $\dim N=\dim M$ to the Lagrangian case, see for example \cite{Fl}. 
\subsubsection*{Discussion of optimality}\label{subsec:disc}
Let $M$ be a manifold, $f:M\to \R$ a Morse function, and $R$ a commutative ring. We denote by $\Crit f\sub M$ the set of critical points of $f$. Recall that $f$ is called $R$-perfect iff $|\Crit f|=\sum_{i=0}^{dim M}b_i(M,R)$. The next result implies that the estimate (\ref{eq:Fix phi}) is sharp if there exists a $\Z_2$-perfect Morse function on the coisotropic submanifold $N$. It actually shows that in this case (\ref{eq:Fix phi}) is sharp, even if the condition (\ref{eq:d phi id M}) is replaced by the much stronger condition that $\phi$ is $C^1$-close to the identity. We denote by $\Ham_\c(M,\om)$ the group of compactly supported Hamiltonian diffeomorphisms of $M$. 
\begin{thm}\label{thm:f} Let $(M,\om)$ be a symplectic manifold, $N\sub M$ a closed regular coisotropic submanifold, $f\Colon N\to\R$ a Morse function, $\iota:M\to\R^{4n}$ an embedding, and $\eps>0$. Then there exists $\phi\in\Ham_\c(M,\om)$ such that $(N,\phi)$ is non-degenerate, and
\[\Fix(\phi,N)=\Crit f,\quad \big\Vert\iota\circ\phi\circ\iota^{-1}-\id\big\Vert_{C^1(\iota(M))}<\eps.\] 
\end{thm}
The proof of this result relies on a normal form theorem for a neighborhood of $N$, which is due to Marle, and on the fact that fast almost periodic orbits of a vector field are constant. It also uses an estimate for the distance between the initial and the end point of a path $x$ in foliation, assuming that these points lie in the same leaf, and that $x$ is tangent to a given horizontal distribution. 

Examples of manifolds admitting a $\Z_2$-perfect Morse function include the following: 
\begin{enui}\item\label{Stiefel} The real, complex and quaternonian Stiefel manifolds. (See \cite{TT}, the remarks after Ex.~3.14 on p.~197, and the definition of tautness on p.~182.) 
\item Compact symmetric spaces that admit a symmetric embedding into Euclidian space. This includes the real, complex and quaternonian Grassmannian. (See \cite{DV} Theorem 1.2 and example 1. on p.~7.) 
\item\label{G T} Quotients $G/T$, where $G$ is a compact connected semi-simple Lie group, and $T\sub G$ is a maximal torus. (This follows for example from \cite{Du} Theorems 4 and 5 on p. 125.) Note that for $G:=\SU(n)$ $G/T$ is diffeomorphic to the manifold of complete flags in $\C^n$. 
\item\label{Ham S 1} Symplectic manifolds that admit a Hamiltonian $S^1$-action whose fixed points are isolated, see for example \cite{GGK}, the theorem on p.~22. 
\item Simply connected closed manifolds of dimension at least 6, whose homology with $\Z$-coefficients is torsion-free. (This follows from \cite{An} Theorem 4.2.4(ii) on p. 112, Definition 4.1.1  on p. 106, and formula (4.2.4) on p. 111.) 
\end{enui}
Note that example (\ref{Ham S 1}) generalizes example (\ref{G T}). Observe also that the product of two manifolds allowing a $\Z_2$-perfect Morse function, has the same property. (This follows from \cite{An} Theorem 4.1.5 on p. 109 and the K\"unneth formula.) 

Consider now $\C^{k\x n}$ with the standard symplectic form $\om_0$, and the Stiefel manifold $V(k,n)\sub \C^{k\x n}$. Then by the next result the condition $C\geq A\big(\C^{k\x n},\om_0,V(k,n)\big)$ in Theorem \ref{thm:leaf} is sharp. We denote by $d_\c$ the compactly supported Hofer distance (see Section \ref{sec:back}). 
\begin{prop}\label{prop:Stiefel} We have $A\big(\C^{k\x n},\om_0,V(k,n)\big)=\pi$. Furthermore, for every $C>\pi$ there exists $\phi\in\Ham_\c(\C^{k\x n},\om_0)$ such that $d_\c(\phi,\id)<C$ and $\phi(V(k,n))\cap V(k,n)=\empty$. 
\end{prop}
\subsubsection*{Arnol'd-Givental conjecture (AGC) for coisotropic submanifolds} Recall that a map from a set to itself is called an involution iff applying it twice yields the identity. Furthermore, a diffeomorphism $\psi$ from a symplectic manifold $(M,\om)$ to itself is called anti-symplectic iff $\psi^*\om=-\om$. The following conjecture naturally generalizes the usual (Lagrangian) AGC to ``product''-coisotropic submanifolds in products of symplectic manifolds.\\

\noindent{\bf Conjecture.} \emph{Let $(M_i,\om_i)$, $i=1,2$ be symplectic manifolds, with $M_1$ closed, and let $L\sub M_2$ be a closed Lagrangian submanifold. Consider the product $M:=M_1\x M_2$ with the symplectic structure $\om:=\om_1\oplus\om_2$, and let $N:=M_1\x L$. Assume that there exists an anti-symplectic involution $\psi\Colon M_2\to M_2$ such that $\Fix(\psi)=L$. Let $\phi\in\Ham(M,\om)$ be such that the pair $(N,\phi)$ is non-degenerate. Then inequality (\ref{eq:Fix phi}) holds.}\\

In the case in which $M_1$ is a point this is the usual (Lagrangian) AGC. (See for example \cite{Fr}, where it is assumed that $M$ is compact.) 
\begin{prop}\label{prop:Arnold} If the \emph{Lagrangian} AGC is true then the same holds for the above Conjecture. 
\end{prop}
\subsection*{An application} 
By definition a presymplectic manifold is a pair $(M,\om)$, where $M$ is a manifold, and $\om$ is a closed two-form on $M$ of constant corank $\corank\om$ (see Section \ref{sec:back}). We say that a presymplectic manifold $(M',\om')$ embeds into a presymplectic manifold $(M,\om)$ iff there exists an embedding $\psi\Colon M'\to M$ such that $\psi^*\om=\om'$. The following question generalizes the symplectic and Lagrangian embedding problems.\\

\noindent {\bf Question B:} Given two presymplectic manifolds, does one of them embed into the other one?\\

Note that in the case $\dim M'+\corank\om'>\dim M+\corank\om$ there does not even exist any immersion $\psi\Colon M'\to M$ satisfying $\psi^*\om=\om'$. (This follows from Proposition \ref{prop:M om} below.) The next result is concerned with the ``critical case'' in which ``$>$'' is replaced by ``$=$'' above. It is a consequence of Theorem \ref{thm:leaf}. A presymplectic manifold $(M,\om)$ is called regular iff its isotropic leaf relation is a closed subset of $M\x M$ and a submanifold.
\begin{cor}\label{cor:presympl} Let $(M,\om)$ be a bounded and aspherical symplectic manifold, and $(M',\om')$ a closed, regular presymplectic manifold of corank $\dim M-\dim M'$. Assume that every compact subset of $M$ can be displaced in a Hamiltonian way, and that $M'$ has a simply-connected isotropic leaf. Then $(M',\om')$ does not embed into $(M,\om)$. 
\end{cor}
\vspace{0.5ex}
\subsubsection*{Examples} 
As an example, let $(X,\si)$ and $(X',\si')$ be symplectic manifolds, the former bounded and aspherical and the latter closed. Let $F$ be a closed simply-connected manifold. Assume that $\dim X+2=\dim X'+2\dim F$. Then the hypotheses of Corollary \ref{cor:presympl} are satisfied with 
\[M:=X\x\R^2,\quad \om:=\si\oplus\om_0,\quad M':=X'\x F,\quad\om':=\si'\oplus0.\] 

As a more specific example, let $(X',\si')$ be a closed aspherical symplectic manifold, and $k\geq2$ and $0\leq\ell\leq k$ be integers. We define
\[(M,\om):=\big(X'\x\R^{2(k-\ell)}\x\R^\ell,\si'\oplus\om_0\oplus0\big),\quad (M',\om'):=\big(X'\x S^k,\si'\oplus0\big).\]
Then $(M',\om')$ does not embed into $(M,\om)$. To see this, observe that every embedding of $(M',\om')$ into $(M,\om)$ gives rise to an embedding of $(M',\om')$ into $\big(X'\x\R^{2k},\si'\oplus\om_0\big)$, by composition with the canonical inclusion $M\to X'\x\R^{2k}$. Hence the statement follows from Corollary \ref{cor:presympl}.

However, in this example there exists an embedding $\psi\Colon M'\to M$ such that $\psi^*[\om]=[\om']$, provided that $\ell<k$. We may for example choose any embedding $\iota\Colon S^k\to \R^{2(k-\ell)}\x\R^\ell$ and define $\psi:=\id_{X'}\x\iota$. Furthermore, if $\ell=0$ then there exists an \emph{immersion} $\psi\Colon M'\to M$ satisfying $\psi^*\om=\om'$. To see this, note that the Whitney map
\[f\Colon S^k\sub \R\x\R^k\to \R^{2k}\iso\C^k,\quad f(a,x):=(1+ai)x\]
is a Lagrangian immersion. (See \cite{ACL}, Example I.4.3, p. 17.) The map $\psi:=\id_{X'}\x f$ has the desired properties. 
\subsection*{Further research} 
A further direction of research is to replace the closeness assumption (\ref{eq:d phi id M}) by a suitable monotonicity assumption. This requires a definition of a Maslov map of the triple $(M,\om,N)$. In a forthcoming article \cite{Zi1} I give such a definition. 
\vspace{0.5ex}
\subsection*{Related results}
In the extreme cases $N=M$ and $\dim N=\dim M/2$ Question A has been investigated a lot. For some references, see for example \cite{MS}, Sec. 9.1., p.~277, and \cite{Gin}, Sec.~1.1.\ p.112. If $(M,\om)$ is a closed symplectic manifold, and $\phi\in\Ham(M,\om)$ is such that every $x\in\Fix(\phi)$ is non-degenerate then Arnol'd \cite{Ar} conjectured that $|\Fix(\phi)|\geq|\Crit f|$ for every Morse function $f:M\to\R$. 

The general coisotropic case was first considered by J.~Moser. He proved that $\big|\Fix(\phi,N)\big|\geq2$ if $M$ is simply connected, $\om$ is exact, and the $C^1$-distance $d^{C^1}(\phi,\id)$ is sufficiently small, see the theorem on p. 19 in \cite{Mos}. (In fact, he showed that $\big|\Fix(\phi,N)\big|$ is bounded below by the Lusternik-Schnirelmann category of $N$, see Proposition 5, p.31 in \cite{Mos}.) A.~Banyaga \cite{Ba} removed the simply connectedness and exactness conditions. Because of the $C^1$-closeness condition these are local results. Global results were first obtained by I.~Ekeland and H.~Hofer \cite{EH,Ho}. For $N$ a closed connected hypersurface in $\R^{2n}$ of restricted contact type they gave several criteria under which $\Fix(\phi,N)\neq\emptyset$, allowing for interesting cases in which $d^{C^1}(\phi,\id)$ is big. For example, in Theorem 1.6 in \cite{Ho} it is assumed that the compactly supported Hofer distance $d_\c(\phi,\id)$ is bounded above by the Ekeland-Hofer capacity $c_{EH}(N)$. Recall here that a coisotropic submanifold $N\sub M$ of codimension $k$ is said to be of contact type iff there exist one-forms $\al_1,\ldots,\al_k$ on $N$ such that $d\al_i=\om$, for $i=1,\ldots,k$, and $\al_1\wedge\cdots\wedge \al_k\wedge\om|_N^{n-k}$ does not vanish anywhere on $N$. Here $\om|_N$ denotes the pullback of $\om$ under the inclusion of $N$ into $M$. $N$ is said to be of restricted contact type iff the $\al_i$'s extend to global primitives of $\om$. D. Dragnev (\cite{Dr}, Theorem 1.3) proved a similar result for general codimension of $N$, replacing $c_{EH}(N)$ by the Floer Hofer capacity of $N$, and assuming that $N$ is only of contact type. 

Generalizing in another direction, V.~Ginzburg proved a version of Hofer's result for subcritical Stein manifolds, replacing $c_{EH}$ by some homological capacity $c_\hom$ (see \cite{Gin}, Theorem 2.9 p. 122). This result in turn was recently extended by B.~G\"urel \cite{Gur} to the coisotropic case (with $c_\hom$ replaced by some constant depending on $N$). For general codimension of $N$, Ginzburg observed that $\Fix(\phi,N)\neq\emptyset$ if the isotropic foliation of $N$ is a fibration (i.e. $N$ is regular) and ``$\phi$ is not far from $\id$ in a suitable sense'', see \cite{Gin}, Example 1.3 p. 113. His argument is based on the fact that in this case the leaf relation is a Lagrangian submanifold of the product $M\x M$, equipped with the symplectic form $\om\oplus(-\om)$. Lately, P.~Albers and U.~Frauenfelder proved that $\Fix(\phi,N)\neq\emptyset$, if $(M,\om)$ is convex at infinity, $N\sub M$ is a closed hypersurface of restricted contact type, and $d_\c(\phi,\id)<A(M,\om,N)$. If in addition the Rabinowitz action functional of the Hamiltonian function generating $\phi$ is Morse, then they showed that $\big|\Fix(\phi,N)\big|\geq\sum_ib_i(N,\Z_2)$. (See Theorems A and B in \cite{AF}.) A problem related to Question A is to find a lower bound on the displacement energy of a coisotropic submanifold. Recent work on this problem other than the one already mentioned has been carried out by E.~Kerman in \cite{Ke}. 

Note that regularity of $N$ and the contact type condition do not imply each other. For example, every Lagrangian submanifold is regular. However, if $N$ is a closed connected Lagrangian submanifold of contact type then it is a torus, see for example \cite{Gin}, Example 2.2 (iv), p. 118. On the other hand, consider $(M,\om):=(\C^2,\om_0)$, fix an irrational number $a>0$, and define $H\Colon \C^2\to \R$ by $H(z,w):=|z|^2+|w|^2/a$. Then the ellipsoid $N:=H^{-1}(1)\sub M$ is a hypersurface of restricted contact type, since the region bounded by $N$ is convex. However, the only compact isotropic leaves are the circles $\big\{(z,0)\,\big|\,|z|^2=1\big\}$ and $\big\{(0,w)\,\big|\,|w|^2=a\big\}$. (The leaves are the integral curves of the Hamiltonian vector field of $H$.) Hence $N$ is not regular. Note that ``restricted contact type'' is a global condition on $(M,\om,N)$, whereas regularity is a condition only on $(N,\om|_N)$. 

If $N$ is of restricted contact type then it is stable (see Definition 2.1 p.~117 in \cite{Gin}). Regularity and stability can be seen as ``dual'' conditions in the following sense. Namely, $(N,\om|_N)$ is regular if and only if it fibers into isotropic submanifolds, whereas it is stable if and only if some neighborhood of $N$ fibers as a family of coisotropic submanifolds containing $N$, see \cite{Gin} Proposition 2.6, p. 120. Observe also that V. Ginzburg constructed a closed hypersurface $N\sub\R^{2n}$ without any closed characteristic, see \cite{Gin}, Example 7.2 p. 158. This means that $A(\R^{2n},\om_0,N)=\infty$. Furthermore, B. G\"urel gave an example of a hypersurface $N\sub\R^4$ such that $A(\R^4,\om_0,N)=\infty$, and for every $\eps>0$ there exists $\phi\in\Ham(M,\om)$ satisfying $\Fix(\phi,N)=\emptyset$ and $d_\c(\phi,\id)<\eps$ (see \cite{Gur}). This shows that one may not completely drop the regularity or stability condition on $N$ if one wants to prove existence of leafwise fixed points. 

Let now $M,\om,M'$ and $\om'$ be as in the hypothesis of Corollary \ref{cor:presympl}. Assume that $(M,\om)$ is the product of some bounded symplectic manifold with $(\R^2,\om_0)$ and that $M'$ is simply connected. Then the statement of the corollary follows from the comments after Example 2.2.8. on pp. 288 and 289 in \cite{ALP}, using Proposition \ref{prop:M om} and Lemma \ref{le:X/R} below. Like the proof of Corollary \ref{cor:presympl}, that argument is based on the fact that the image of the map $\iota_N$ defined in (\ref{eq:iota N wt N}) is a Lagrangian submanifold of $M\x N_\om$. However, since it does not involve the Key Lemma \ref{le:A M om N}, the assumption that $M'$ is simply connected is needed there. On the other hand, if $\om$ is exact then Corollary \ref{cor:presympl} can be deduced from Example 1.7, p.115 in \cite{Gin}, using again Proposition \ref{prop:M om} and Lemma \ref{le:X/R}. Furthermore, if the presymplectic manifold $(M',\om')$ is stable then a similar non-embedding result can be deduced from Theorem 2.7 (ii) p. 121 in \cite{Gin}.

\subsection*{Organization of the article}
Section \ref{sec:back} contains some background on foliations, presymplectic manifolds, coisotropic submanifolds, leafwise fixed points, and Hamiltonian diffeomorphisms. In this section, the linear holonomy of a foliation along a path in a leaf, and based on this, non-degeneracy of a pair $(N,\phi)$, are defined. In Section \ref{sec:reduction} Chekanov's theorem is restated (Theorem \ref{thm:L}), and the relevant properties of the map $\iota_N$ and the subset $\wt N\sub\wt M$ (as in (\ref{eq:iota N wt N})) are established (Lemmas \ref{le:N} and \ref{le:A M om N}). Based on this, the main results are proven in Section \ref{sec:proofs}. Appendix \ref{subsec:sympl geo} contains some background about presymplectic geometry, on the embedding of a smooth fiber bundle over a symplectic base into its vertical cotangent bundle, and three other elementary results from symplectic geometry. In Appendix \ref{subsec:fol} the result is proven that is used in the definition of linear holonomy. Furthermore, an estimate for a tangent path of a horizontal distribution in a foliation is proven. Finally, Appendix \ref{subsec:further} contains results about smooth structures on the quotient set of an equivalence relation, fast almost periodic orbits of a vector field, and a measure theoretic lemma. 

\subsection*{Acknowledgments}
I am specially indebted to Yael Karshon for numerous enlightening discussions and her continuous support and encouragement. She also made me aware that the conclusion of Corollary \ref{cor:presympl} follows from an easy cohomological argument if $(M,\om):=(\R^{2n},\om_0)$ and $(M',\om'):=\big(X'\x F,\si'\oplus0\big)$, with $X'$ closed, $\dim X'>0$ and $\si'$ symplectic. I would also like to thank Viktor Ginzburg, Chris Woodward and Masrour Zoghi for interesting conversations, and Urs Frauenfelder for useful hints. It was Chris Woodward from whom I learned about the construction of the Lagrangian submanifold $\wt N\sub\wt M$ in the case of a Hamiltonian Lie group action, with $N:=\mu^{-1}(0)$. 

\section{Background}\label{sec:back}
\subsection*{Notation, manifolds}
We denote by $\N$ the positive integers, by $\D,S^1\sub\R^2$ the closed unit disk and the unit circle, and for $r>0$ by $B_r\sub\R^2$ the open ball of radius $r$. For two vector spaces $V$ and $V'$ and a linear map $\Psi\Colon V'\to V$ we denote by $\ker\Psi$ and $\im\Psi$ its kernel and image, and by $\Psi^*\Colon V^*\to{V'}^*$ its adjoint map. Let $M$ be a set. By a \emph{smooth structure} on $M$ we mean a maximal smooth ($C^\infty$) atlas $\A$ of charts $\phi\Colon U\sub M\to \R^n$. (Hence $M$ does not have any boundary.) Assume that $M$ is equipped with a smooth structure. We denote by $C^\infty(M,\R)$ and $C^\infty_\c(M,\R)$ the set of smooth and compactly supported smooth functions, respectively. We call $(M,\A)$ a \emph{manifold} iff the topology on $M$ induced by $\A$ is Hausdorff and second countable. \emph{Submanifolds} of $M$ are by definition embedded. For a smooth time-dependent vector field $X$ on a manifold $M$ and $t\in\R$ we denote by $\phi_X^t:M\to M$ its time-$t$-flow (if it exists).
\subsection*{Foliations, regularity, and linear holonomy}
We recollect some basic definitions and facts about foliations. For more details, see for example the book \cite{MM}. The definition of linear holonomy given below will be needed to define non-degeneracy of a pair $(N,\phi)$ as in section \ref{sec:main}. I am not aware of a reference in which the linear holonomy is defined in precisely this way. However, the basic idea of its definition is standard, see for example \cite{MM}.

Let $M$ be a set, $0\leq k\leq n$ integers, $U,U'\sub M$ subsets, and $\phi\Colon U\to\R^n$, $\phi'\Colon U'\to \R^n$ injective maps. We denote by $\pi_1\Colon \R^n=\R^{n-k}\x\R^k\to\R^{n-k}$ the canonical projection onto the first factor. We call $(U,\phi)$ and $(U',\phi')$ $(n,k)$-compatible iff $\phi(U\cap U')\sub\R^n$ is open, $\phi'\circ\phi^{-1}\Colon \phi(U\cap U')\to\phi'(U\cap U')$ is a diffeomorphism, and for every $\xi\in\R^{n-k}$ the map 
\[\pi_1\circ \phi'\circ\phi^{-1}(\xi,\cdot)\Colon \big\{\eta\in\R^k\,\big|\,(\xi,\eta)\in \phi(U\cap U')\big\}\to \R^{n-k}\]
is locally constant. We define an $(n,k)$-atlas on $M$ to be a set $\A$ of pairs $(U,\phi)$ as above, such that $\bigcup_{(U,\phi)\in\A}U=M$ and each two pairs in $\A$ are $(n,k)$-compatible. An $(n,k)$-foliation on $M$ is defined to be a maximal (with respect to inclusion) $(n,k)$-atlas on $M$. Let $\F$ be an $(n,k)$-foliation on $M$. We endow $M$ with the smooth structure induced by $\F$, and for $x\in M$, we define 
\[T_x\F:=d\phi(x)^{-1}\big(\{0\}\x\R^k\big)\sub T_xM,
\]
where $(U,\phi)\in\F$ is a chart such that $x\in U$. We define the leaf through a point $x_0\in M$ to be the set
\[\F_{x_0}:=\big\{x(1)\,\big|\,x\in C^\infty([0,1],M):\,x(0)=x_0,\,\dot x(t)\in T_{x(t)}\F,\,\forall t\big\}\sub M.\]
The leaf relation is defined to be the set  
\[R^\F:=\big\{(x_0,x_1)\in M\x M\,\big|\,x_1\in\F_{x_0}\big\}.\] 
It is an equivalence relation on $M$. The collection of the subspaces $T_x\F$, with $x\in M$, is an involutive distribution $T\F$ on $M$, called the tangent bundle to $\F$. We denote by $N\F:=TM/T\F$ the normal bundle, and by $\pr^\F\Colon TM\to N\F$ the canonical projection, and for $x\in M$, we write $N_x\F:=(N\F)_x$. 

Let now $(M,\A)$ be a manifold. By a foliation on $(M,\A)$ we mean a foliation $\F$ on $M$ that induces $\A$. In this case we call the pair $(M,\F)$ a foliated manifold. Note that if $E$ is an involutive distribution on $M$ then by Frobenius' Theorem there exists a unique foliation $\F^E$ on $M$ such that $E_x=d\phi(x)^{-1}\big(\{0\}\x\R^k\big)$ for every chart $(U,\phi)\in\F^E$ for which $x\in U$. We call a foliated manifold $(M,\F)$ \emph{regular} iff $R^\F$ is a closed subset and submanifold of $M\x M$. By Lemma \ref{le:X/R} below this holds if and only if there exists a manifold structure on the quotient $M/R^\F$ such that the canonical projection from $M$ to $M/R^\F$ is a submersion. Furthermore, such a structure is unique. If $\F$ is regular then the leaves are closed subsets and submanifolds of $M$. If $(M,\F)$ is a foliated manifold and $H\sub TM$ a distribution then we call $H$ \emph{($\F$-)horizontal} iff for every $x\in M$ we have $T_xM=H_x\oplus T_x\F$. 

Let $(M,\F)$ be a foliated manifold, $F$ a leaf of $\F$, $a\leq b$, and $x\in C^\infty([a,b],F)$ a path. The \emph{linear holonomy of $\F$ along $x$} is a linear map $\hol^\F_x:N_{x(a)}\F\to N_{x(b)}\F$. Its definition is based on the following result. 
\begin{prop}\label{prop:hol} Let $M,\F,F,a,b$ and $x$ be as above, $N$ a manifold, and $y_0\in N$. Then the following statements hold.
\begin{enua}\item\label{prop:hol:u} For every linear map $T\Colon T_{y_0}N\to T_{x(a)}M$ there exists a map $u\in C^\infty([a,b]\x N,M)$ such that
\begin{eqnarray}\label{eq:u 0 x}&u(\cdot,y_0)=x,\quad \F_{u(t,y)}=\F_{u(a,y)},\,\forall t\in[a,b],\,y\in N,&\\
\label{eq:d u a y}&d(u(a,\cdot))(y_0)=T.& 
\end{eqnarray}
\item\label{prop:hol:u u'} Let $u,u'\in C^\infty([a,b]\x N,M)$ be maps satisfying (\ref{eq:u 0 x}), such that
\begin{equation}\label{eq:u u' a}\pr^\F d(u(a,\cdot))(y_0)=\pr^\F d(u'(a,\cdot))(y_0).\end{equation}
Then $\pr^\F d(u(b,\cdot))(y_0)=\pr^\F d(u'(b,\cdot))(y_0)$. 
\end{enua}
\end{prop}
We choose a linear map $T:N_{x(a)}\F\to T_{x(a)}M$, such that $\pr^\F T=\id_{N_{x(a)}\F}$. Furthermore, we define $N:=N_{x(a)}\F$ and $y_0:=0$, and we choose a map $u\in C^\infty\big([a,b]\x N_{x(a)}\F,M\big)$ such that (\ref{eq:u 0 x}) and (\ref{eq:d u a y}) hold. Here in (\ref{eq:d u a y}) we canonically identified $T_0\big(N_{x(a)}\F\big)=N_{x(a)}\F$. We define
\begin{equation}\label{eq:hol F x}\hol^\F_x:=\pr^\F d(u(b,\cdot))(0)\Colon N_{x(a)}\F(=T_0(N_{x(a)}\F))\to N_{x(b)}\F.\end{equation}
It follows from Proposition \ref{prop:hol} that this map is well-defined. It can be viewed as the linearization of the holonomy of a foliation as defined for example in Sec. 2.1 in the book \cite{MM}. 

\subsection*{Presymplectic manifolds and symplectic quotients}
By a presymplectic vector space we mean a real vector space $V$ together with a skew-symmetric bilinear form $\om$. Let $(V,\om)$ be such a pair. For every linear subspace $W\sub V$ we denote by $W^\om:=\big\{v\in V\,\big|\,\om(v,w)=0,\,\forall w\in W\big\}$ its $\om$-complement. The subspace $W$ is called coisotropic iff $W^\om\sub W$. We define $\corank\om:=\dim V^\om$. A \label{presympl} presymplectic structure on a manifold $M$ is a closed two-form $\om$ on $M$, such that $\corank\om_x$ does not depend on $x\in M$. This number is called the corank of $\om$. Let $(M,\om)$ be a presymplectic manifold. The \emph{isotropic distribution} $TM^\om=\big\{(x,v)\,\big|\,x\in M,\,v\in T_xM^\om\big\}\sub TM$ is involutive. (This follows for example as in the proof of Lemma 5.33 in the book \cite{MS}.) We call $\F^{TM^\om}$ the isotropic (or characteristic) foliation on $M$. We call $(M,\om)$ \emph{regular} \label{M om reg} iff $\F^{TM^\om}$ is regular. Assume that $(M,\om)$ is regular. We denote by $M_\om$ the set of isotropic leaves of $M$, by $\pi_{M,\om}:M\to M_\om$ the canonical projection, and by $\A_{M,\om}$ the unique manifold structure on $M_\om$ such that $\pi_{M,\om}$ is a submersion. There exists a unique symplectic form $\om_M$ on $M_\om$ satisfying $\pi_{M,\om}^*\om_M=\om$. The triple $\big(M_\om,\A_{M,\om},\om_M\big)$ is called the \emph{symplectic quotient} of $(M,\om)$. If $M$ is also closed then by a result by C.~Ehresmann the quadruple $\big(M,M_\om,\A_{M,\om},\pi_{M,\om}\big)$ is a smooth (locally trivial) fiber bundle, see the proposition on p. 31 in \cite{Eh}. 

On the other hand, given a smooth fiber bundle $(E,B,\pi)$ and a symplectic form $\si$ on $B$, the pair $(E,\pi^*\si)$ is a regular presymplectic manifold. Let $(M,\om)$ be a presymplectic manifold, and $H\sub TM$ a distribution. We call $H$ \emph{($\om$-)horizontal} iff it is $\F^{TM^\om}$-horizontal. This means that for every $x\in M$ we have $T_xM=H_x\oplus T_xM^\om$. Assume that $H$ is horizontal. This gives rise to a closed two-form $\Om_{\om,H}$ on the manifold $(TM^\om)^*$, as follows. For $x\in M$ we denote by $\pr^H_x:T_xM\to T_xM^\om$ the linear projection along the subspace $H_x\sub T_xM$, and we define 
\begin{equation}\label{eq:iota H TM}\iota_H:(TM^\om)^*\to T^*M,\quad \iota_H(x,\al):=\big(x,\al\circ\pr^H_x\big).\end{equation}
We denote by $\pi:(TM^\om)^*\to M$ the canonical projection, and by $\om_\can$ the canonical symplectic form on $T^*M$. We define 
\begin{equation}\label{eq:Om om H}\Om_{\om,H}:=\pi^*\om+\iota_H^*\om_\can.
\end{equation}
By a result by C.-M.~Marle there exists an open neighborhood of the zero section on which $\Om_{\om,H}$ is non-degenerate, see Proposition 3.2 in \cite{Ma}.
\vspace{0.5ex}
\subsection*{Coisotropic submanifolds and leafwise fixed points}
Let $(M,\om)$ be a presymplectic manifold. A submanifold $N\sub M$ is called coisotropic iff for every $x\in N$ the subspace $T_xN\sub T_xM$ is coisotropic. This holds if and only if the restriction $\om|_N$ of $\om$ to $N$ (i.e. the pull-back under the inclusion map) is a presymplectic form satisfying $\dim N+\corank\om|_N=\dim M+\corank\om$. (This follows from Proposition \ref{prop:M om} below.) If $N$ is coisotropic then $2\dim N\geq\dim M+\corank\om$. In the extreme case $\corank\om=0$ (i.e. $\om$ symplectic) and $\dim N=\dim M/2$ the submanifold $N$ is called Lagrangian. As an example, let $F$ be a manifold, and $(X,\si)$ a symplectic manifold. We denote by $\om_\can$ the canonical two-form on $T^*F$, and define $\om:=\si\oplus\om_\can$. Then $X\x F$ is a coisotropic submanifold of $(X\x T^*F,\om)$. As another example, every hypersurface (i.e. real codimension one submanifold) of a symplectic manifold is coisotropic. 

Let $N\sub M$ be a coisotropic submanifold. For a point $x\in N$ we denote by $N_x:=N^\om_x\sub N$ the isotropic leaf through $x$. Furthermore, we denote by $N_\om$ the set of all isotropic leaves of $N$, and by $\pi_N:N\to N_\om$ the canonical projection. We define the \emph{action spectrum} and the \emph{minimal area} of $(M,\om,N)$ as
\begin{equation}\label{eq:S}S(M,\om,N):=\left\{\int_\D u^*\om\,\bigg|\,u\in C^\infty(\D,M):\,\exists F\in N_\om:\, u(S^1)\sub F\right\},\end{equation}
\begin{equation}\label{eq:A M om N} A(M,\om,N):=\inf\big(S(M,\om,N)\cap (0,\infty)\big)\in[0,\infty].
\end{equation}
Furthermore, we denote the linear holonomy $\hol^{\F^{TN^\om}}$ by $\hol^{\om,N}$. We call $N$ \label{N reg} \emph{regular} iff the presymplectic manifold $(N,\om|_N)$ is regular (i.e.  the foliation $\F^{TN^\om}$ on $N$ is regular). Note that such an $N$ is sometimes called ``fibering''. 

Let $\phi\Colon M\to M$ be a map. We say that a point $x\in N$ is leafwise fixed under $\phi$ iff $\phi(x)\in N_x$. We denote by $\Fix(\phi,N)=\Fix(\phi,N,\om)$ the set of such points. Assume now that $\corank\om=0$, i.e. that $\om$ is symplectic, and that $\phi$ is smooth. We call $(N,\phi,\om)$ (or simply $(N,\phi)$) \label{non-deg} \emph{non-degenerate} iff the following holds. For $x_0\in N$ we denote by $\pr_{x_0}:T_{x_0}N\to T_{x_0}N/(T_{x_0}N)^\om$ the canonical projection. Let $F\sub N$ be an isotropic leaf, and $x\in C^\infty([0,1],F)$ a path. Assume that $\phi(x(0))=x(1)$, and let $v\in T_{x(0)}N\cap T_{x(0)}\phi^{-1}(N)$ be a vector. Then $v\neq0$ implies that 
\begin{equation}\label{eq:hol x pr}\hol^{\om,N}_x\pr_{x(0)}v\neq \pr_{x(1)}d\phi(x(0))v.\end{equation}
Note that in the case $N=M$ this condition means that for every fixed point $x_0$ of $\phi$ the differential $d\phi(x_0)$ does not have 1 as an eigenvalue. Furthermore, in the case that $N$ is Lagrangian the condition means that every connected component $N'\sub N$ intersects $\phi(N')$ transversely.

\subsection*{Hamiltonian diffeomorphisms and Hofer distance}
Let $M$ be a manifold, and $H\in C^\infty\big([0,1]\x M,\R\big)$. We abbreviate $H_t:=H(t,\cdot)$. The Hofer semi-norm of $H$ is defined to be
\begin{equation}\label{eq:Vert H}\Vert H\Vert:=\Vert H\Vert_M:=\int_0^1\big(\sup_MH_t-\inf_MH_t\big)\in[0,\infty].\end{equation}
It follows from Lemma \ref{le:f} in the appendix that the function $[0,1]\ni t\mapsto \sup_MH_t-\inf_MH_t\in[0,\infty]$ is Borel measurable. Therefore, the integral (\ref{eq:Vert H}) is well-defined. Let now $(M,\om)$ be a symplectic manifold. For a function $H\in C^\infty(M,\R)$ we define the vector field $X_H$ generated by $H$ via the formula $dH=\om(X_H,\cdot)$. For $t\in\R$ we denote by $\phi_H^t$ the time-$t$-flow of the family $(X_{H_s})_{s\in\R}$ (if it exists). A diffeomorphism $\phi\Colon M\to M$ is called Hamiltonian iff there exists a function $H\in C^\infty\big([0,1]\x M,\R\big)$ such that $\phi_H^1$ exists and equals $\phi$. We denote by \label{Ham M om} $\Ham(M,\om)$ the group of all Hamiltonian diffeomorphisms. For $\phi,\psi\in\Ham(M,\om)$ we define the Hofer distance $d^{M,\om}(\phi,\psi)=d(\phi,\psi)$ to be
\begin{equation}\label{eq:d phi psi}d(\phi,\psi):=\inf\big\{\Vert H\Vert\,\big|\,H\in C^\infty\big([0,1]\x M,\R\big):\,\phi_H^1=\psi^{-1}\circ\phi\big\}\in[0,\infty].\end{equation}
We denote by $\Ham_\c(M,\om)\sub \Ham(M,\om)$ the subgroup of all diffeomorphisms that are generated by some compactly supported function $H:[0,1]\x M\to \R$. We define the \emph{compactly supported Hofer distance} of $\phi,\psi\in\Ham_\c(M,\om)$ to be
\begin{equation}\label{eq:d c}d_\c(\phi,\psi):=\inf\big\{\Vert H\Vert\,\big|\,H\in C^\infty_\c\big([0,1]\x M,\R\big):\,\psi^{-1}\circ\phi=\phi_H^1\big\}.\end{equation}

\subsection*{Geometric boundedness}
Let $(M,\om)$ be a symplectic manifold, and $J$ an $\om$-compatible almost complex structure on $M$. We call $(M,\om,J)$ (geometrically) bounded iff the Riemannian metric $g_{\om,J}:=\om(\cdot,J\cdot)$ is complete with bounded sectional curvature and injectivity radius bounded away from 0. We call $(M,\om)$ \label{bdd} (geometrically) bounded iff there exists an almost complex structure $J$ such that $(M,\om,J)$ is bounded. Examples are closed symplectic manifolds, cotangent bundles of closed manifolds, and symplectic vector spaces. For an almost complex manifold $(M,J)$ and a totally real submanifold $N\sub M$ we define 
\begin{eqnarray}\label{eq:A S 2}&A_{S^2}(M,J):=\inf\left\{\int_{S^2}u^*\om\,\big|\,u\Colon S^2\to M\,J\textrm{-holomorphic}\right\},&\\
\nn &A_\D(M,N,J):=\inf\left\{\int_\D u^*\om\,\big|\,u\Colon \D\to M\,J\textrm{-holomorphic, }u(S^1)\sub N\right\}.&
\end{eqnarray}
Let $(M,\om)$ be a bounded symplectic manifold, and $L\sub M$ a Lagrangian submanifold. We define
\begin{equation}\label{eq:A b}A_\b(M,\om,L):=\sup\big\{\min\big\{A_{S^2}(M,J),A_\D(M,L,J)\big\}\big\},\end{equation}
where the supremum is taken over all $\om$-compatible almost complex structures $J$ on $M$ such that $(M,\om,J)$ is bounded.
\section{Reduction to the Lagrangian case}\label{sec:reduction}
The proof of Theorem \ref{thm:leaf} is based on the following result, which is a reformulation of the Main Theorem in \cite{Ch}. Recall the definition (\ref{eq:A M om N}) of $A(M,\om,N)$. 
\begin{thm}\label{thm:L} Let $(M,\om)$ be a bounded symplectic manifold and $L\sub M$ a closed Lagrangian submanifold. Then there exists a constant $C\in(0,\infty]$ such that $C\geq A(M,\om,L)$ and the following holds. If $\phi\in\Ham(M,\om)$ satisfies 
\begin{equation}\label{eq:d phi id L}d(\phi,\id)<C, 
\end{equation} 
and $\phi(L)\pitchfork L$ (i.e.  $\phi(L)$ and $L$ intersect transversely), then
\begin{equation}\label{eq:L phi L}\#(L\cap\phi(L))\geq\sum_i b_i(L,\Z_2).
\end{equation}
\end{thm}
For the convenience of the reader, let us recall Chekanov's theorem:
\begin{thm}[\cite{Ch}]\label{thm:Ch}Let $(M,\om)$ be a bounded symplectic manifold, $L\sub M$ a closed Lagrangian submanifold, and $\phi\in\Ham_\c(M,\om)$. If $d_\c(\phi,\id)<A_\b(M,\om,L)$ (defined as in (\ref{eq:A b})) and $\phi(L)\pitchfork L$ then (\ref{eq:L phi L}) holds.
\end{thm}
\begin{proof}[Proof of Theorem \ref{thm:L}]\setcounter{claim}{0} \setcounter{claim}{0} Let $M,\om$ and $L$ be as in the hypothesis. We may assume without loss of generality that $M$ and $L$ are connected. We define $C:=A_\b(M,\om,L)$. By quantization of energy for pseudo-holomorphic spheres and disks, this number is positive. 
\begin{claim}\label{claim:A A} We have $C\geq A(M,\om,L)$.
\end{claim} 
\begin{proof}[Proof of Claim \ref{claim:A A}] Let $J$ be as in the definition of boundedness of $(M,\om)$. Then $A_\D(M,L,J)\geq A(M,\om,L)$. Furthermore, let $u\Colon S^2\to M$ be a $J$-holomorphic map. Since $M$ is connected, there exists a smooth map $v\Colon S^2\iso\C\cup\{\infty\}\to M$ that is smoothly homotopic to $u$ and satisfies $v(\infty)\in L$. We choose a smooth map $f\Colon \D\to S^2$ that maps the interior $B_1\sub \D$ diffeomorphically and in an orientation preserving way onto $\C$. Then the map $v\circ f\Colon \D\to M$ is smooth and satisfies $v\circ f(S^1)\sub L$. Furthermore, 
\[\int_{S^2}u^*\om=\int_{S^2}v^*\om=\int_\D (v\circ f)^*\om.\]
It follows that the set of numbers occurring in (\ref{eq:A S 2}) is contained in $S(M,\om,N)$ (as defined in (\ref{eq:S})), and hence $A_{S^2}(M,J)\geq A(M,\om,L)$. Claim \ref{claim:A A} follows.
\end{proof}
Let $\phi\in\Ham(M,\om)$ be such that condition (\ref{eq:d phi id L}) is satisfied and $\phi(L)\pitchfork L$. Applying Lemma \ref{le:H} below with $K:=L$ there exists a function $H\in C^\infty_\c\big([0,1]\x M,\R\big)$ such that (\ref{eq:H}) holds. It follows that $d_\c(\phi_H^1,\id)<C$. Furthermore, by the first condition in (\ref{eq:H}) we have $\phi_H^1(L)=\phi(L)$, and hence $\phi_H^1(L)\pitchfork L$. Therefore, the hypotheses of Theorem \ref{thm:Ch} are satisfied with $\phi$ replaced by $\phi_H^1$. Inequality (\ref{eq:L phi L}) follows from the conclusion of this theorem. This proves Theorem \ref{thm:L}.
\end{proof}
Let now $(M,\om)$ be a symplectic manifold and $N\sub M$ a coisotropic submanifold. Recall that $N_\om$ denotes the set of isotropic leaves of $N$, and $\pi_N:N\to N_\om$ the canonical projection. We abbreviate $\pi:=\pi_N$. We define $\wt M:=M\x N_\om$, and $\iota_N$ and $\wt N$ as in (\ref{eq:iota N wt N}). For a map $\phi\Colon M\to M$ we define 
\begin{equation}\label{eq:wt phi}\wt\phi:=\phi\x\id_{N_\om}\Colon \wt M\to \wt M.\end{equation} 
Assume that $N$ is regular. Then we denote by $\big(N_\om,\A_{N,\om},\om_N\big)$ the symplectic quotient of $(N,\om|_N)$. We equip $\wt M$ with the manifold structure determined by the manifold structure on $M$ and $\A_{N,\om}$. Furthermore, we define $\wt\om:=\om\oplus(-\om_N)$. This is a symplectic form on $\wt M$. 
\begin{lemma}\label{le:N} Let $(M,\om)$ be a symplectic manifold, and $N\sub M$ a connected coisotropic submanifold. 
\begin{enua}
\item\label{le:N:phi leaf} For every map $\phi\Colon M\to M$ we have 
\[\wt\phi\circ\iota_N(\Fix(\phi,N))=\wt N\cap\wt\phi(\wt N).\]

\noindent Assume now also that $N$ is regular. Then:
\item\label{le:N:wt N} The map $\iota_N$ is a Lagrangian embedding with respect to $\wt\om$.
\item\label{le:N:non-deg} If $\phi\Colon M\to M$ is a diffeomorphism then the pair $(N,\phi)$ is non-degenerate if and only $\wt\phi(\wt N)\pitchfork \wt N$.
 \end{enua}
\end{lemma}
For the proof of Lemma \ref{le:N} we need the following.
\begin{lemma}\label{le:F x} Let $(M,\om)$ be a symplectic manifold, $N\sub M$ a regular coisotropic submanifold, $\phi:M\to M$ a diffeomorphism, $F\sub N$ a leaf, and $x\in C^\infty([0,1],F)$ a path. Assume that $x(1)=\phi\circ x(0)$, and let $v\in T_{x(0)}N\cap T_{x(0)}\phi^{-1}(N)$. Then (\ref{eq:hol x pr}) is equivalent to $\pi_*\phi_*v\neq\pi_*v$. 
\end{lemma}
\begin{proof}[Proof of Lemma \ref{le:F x}]\setcounter{claim}{0} We fix $y\in N$. By Lemma \ref{le:X/R}(\ref{le:X/R:ker}) below we have $\ker d\pi(y)=T_yN^\om$. Hence we may define
\[\Phi_y\Colon (T_yN)_\om=T_yN/T_yN^\om\to T_{N_y}(N_\om),\quad \Phi_y(w+T_yN^\om):=d\pi(y)w.\]
Since $d\pi(y)\Colon T_yN\to (T_yN)_\om$ is surjective, $\Phi_y$ is an isomorphism. Furthermore, $\Phi_y\pr_y=d\pi(y)$, where $\pr_y:T_yN\to(T_yN)_\om$ denotes the canonical projection. Hence Lemma \ref{le:X/R}(\ref{le:X/R:hol}) below implies that $d\pi(x(0))=\Phi_{x(1)}\hol^{\om,N}_x\pr_{x(0)}$. It follows that 
\[(\pi_*\phi_*-\pi_*)v=\Phi_{x(1)}\big(\pr_{x(1)}\phi_*-\hol^{\om,N}_x\pr_{x(0)}\big)v.\]
Since $\Phi_{x(1)}$ is an isomorphism, the statement of Lemma \ref{le:F x} follows.
\end{proof}
\begin{proof}[Proof of Lemma \ref{le:N}]\setcounter{claim}{0} {\bf Statement (\ref{le:N:phi leaf})} follows from the definition of a leafwise fixed point. 

{\bf We prove (\ref{le:N:wt N}).} That $\iota_N$ is an injective Lagrangian immersion follows immediately from the definitions. To see that it is an open map onto its image $\wt N$, let $U\sub N$ be open. We choose an open subset $V\sub M$ such that $V\cap N=U$, and denote by $\pi_1\Colon \wt M=M\x N_\om\to M$ the projection onto the first factor. Then $\iota_N(U)$ is the intersection of $\wt N$ with the open subset $\pi_1^{-1}(V)\sub \wt M$, hence it is relatively open in $\wt N$. This proves (\ref{le:N:wt N}).

{\bf We prove (\ref{le:N:non-deg}).} Assume that $x\in \Fix(\phi,N)$ and denote $\wt x:=\wt\phi\circ \iota_N(x)$. Note that by assertion (\ref{le:N:phi leaf}) $\wt x\in \wt N\cap\wt\phi(\wt N)$. 
\begin{claim}\label{claim:wt phi wt N} $\wt N$ and $\wt\phi(\wt N)$ intersect transversely at $\wt x$ if and only if 
\[0\neq v\in T_{x}N\cap T_{x}\phi^{-1}(N)\Then\pi_*\phi_*v\neq\pi_*v.\]
\end{claim}
\begin{proof}[Proof of Claim \ref{claim:wt phi wt N}] For $y\in N$ we have
\begin{equation}\label{eq:T iota}T_{\iota_N(y)}\wt N={\iota_N}_*T_yN=\big\{(v,\pi_*v)\,\big|\,v\in T_yN\big\}.\end{equation}
Setting $y:=x$, it follows that
\begin{equation}\label{eq:T wt x '}T_{\wt x}\wt\phi(\wt N)=\wt\phi_*{\iota_N}_*T_{x}N=\big\{(\phi_*v,\pi_*v)\,\big|\,v\in T_{x}N\big\}.\end{equation}
On the other hand, since $x\in\Fix(\phi,N)$, we have $\iota_N\circ\phi(x)=\wt\phi\circ\iota_N(x)=\wt x$. Therefore, applying (\ref{eq:T iota}) with $y:=\phi(x)$, and combining with (\ref{eq:T wt x '}), we obtain
\[T_{\wt x}\wt N\cap T_{\wt x}\wt\phi(\wt N)=\big\{\big(\phi_*v,\pi_*v\big)\,\big|\,v\in T_xN:\,\phi_*v\in T_{\phi(x)}N,\,\pi_*\phi_*v=\pi_*v\big\}.\]
Claim \ref{claim:wt phi wt N} follows from this.
\end{proof}
Assume now that $(N,\phi)$ is non-degenerate. Let $\wt x_0\in\wt N\cap\wt\phi(\wt N)$. By assertion (\ref{le:N:phi leaf}) there exists $x_0\in \Fix(\phi,N)$ such that $\wt\phi\circ\iota_N(x_0)=\wt x_0$. We choose a path $x\in C^\infty([0,1],N_{x_0})$ such that $x(0)=x_0$ and $x(1)=\phi(x_0)$. If $0\neq v\in T_{x_0}N\cap T_{x_0}\phi^{-1}(N)$ then by (\ref{eq:hol x pr}) and Lemma \ref{le:F x} we have $\pi_*\phi_*v\neq\pi_*v$. Therefore by Claim \ref{claim:wt phi wt N} with $x$ replaced by $x_0$ the manifolds $\wt\phi(\wt N)$ and $\wt N$ intersect transversely at $\wt x_0$. It follows that $\wt\phi(\wt N)\pitchfork\wt N$.

Conversely, assume now that $\wt\phi(\wt N)\pitchfork\wt N$. Let $F\sub N$ be a leaf, and $x\in C^\infty([0,1],F)$ a path, and assume that $x(1)=\phi\circ x(0)$, and $0\neq v\in T_{x(0)}N\cap T_{x(0)}\phi^{-1}(N)$. By Claim \ref{claim:wt phi wt N} we obtain $\pi_*\phi_*v\neq\pi_*v$. Therefore, by Lemma \ref{le:F x} the inequality (\ref{eq:hol x pr}) is satisfied. It follows that $(N,\phi)$ is non-degenerate. This proves (\ref{le:N:non-deg}) and completes the proof of Lemma \ref{le:N}.
\end{proof}
\begin{lemma}[Key Lemma]\label{le:A M om N} Let $(M,\om)$ be a symplectic manifold, $N\sub M$ a closed, regular coisotropic submanifold, and let $\wt M,\wt\om$ and $\wt N$ be defined as in (\ref{eq:wt M wt om},\ref{eq:iota N wt N}). Then 
\begin{equation}\label{eq:A om N}A(M,\om,N)= A(\wt M,\wt \om,\wt N).\end{equation}
\end{lemma}
\begin{proof}[Proof of Lemma \ref{le:A M om N}]\setcounter{claim}{0} In order to show that (\ref{eq:A om N}) with ``$=$'' replaced by ``$\geq$'' holds, let $u\in C^\infty(\D,M)$ be a map such that there exists a leaf $F\sub N$ satisfying $u(S^1)\sub F$. Then the map 
\[\wt u\Colon \D\to \wt M=M\x N_\om,\quad \wt u(z):=(u(z),N_{u(1)})\]
satisfies $\int_\D\wt u^*\wt\om=\int_\D u^*\om$. The inequality ``$\geq$'' in (\ref{eq:A om N}) follows.

To show the opposite inequality, let $\wt u=(v,w')\in C^\infty(\D,\wt M)$ be a map such that $\wt u(S^1)\sub\wt N$. It suffices to prove that there exists a map $u\in C^\infty(\D,M)$ such that $u(S^1)$ is contained in an isotropic leaf of $N$, and
\begin{equation}\label{eq:int D u om}\int_\D u^*\om=\int_\D \wt u^*\wt\om.\end{equation}
To see this, we choose a smooth map $\rho\Colon [0,1]\to [0,1]$ such that 
\begin{eqnarray*}&\rho(1/2)=1,\quad \rho(r)=r,\,\forall r\in[0,1/4],\quad\rho(r)=1-r,\,\forall r\in [3/4,1],&\\
&\rho'(r)>0,\,\forall r\in (0,1/2),\quad \rho'(r)<0,\,\forall r\in(1/2,1),&
\end{eqnarray*}
and all derivatives of $\rho$ vanish at $1/2$. We define $\phi\Colon \D\to\D$ by $\phi(rz):=\rho(r)z$, for $r\in[0,1]$ and $z\in S^1$. 
\begin{claim}\label{claim:u D M} There exists a smooth map $u\Colon \D\to M$ such that
\[u(z)=v\circ\phi(z),\quad\textrm{if }|z|\leq1/2,\]
\begin{equation}\label{eq:u D B}u(z)\in N,\quad \pi \circ u(z)=w'\circ \phi(z),\quad\textrm{if }1/2<|z|\leq1.\end{equation}
\end{claim}
\begin{proof}[Proof of Claim \ref{claim:u D M}] We define $f'\Colon [0,1]\x S^1\to N_\om$ by $f'(t,z):=w'(tz)$. For every $z\in S^1$ we have by assumption $\wt u(z)\in\wt N$, i.e. $\pi\circ v(z)=w'(z)=f'(1,z)$. Hence $f'$ is a smooth homotopy in $N_\om$, ending at the map $\pi\circ v|_{S^1}$. Since $N$ is closed and the projection $\pi\Colon N\to N_\om$ is a submersion, results by Ehresmann imply that there exists a smooth map $f\Colon [0,1]\x S^1\to N$ such that $\pi\circ f=f'$ and $f(1,\cdot)=v|_{S^1}$. (See \cite{Eh}, the proposition on p. 31 and the second proposition on p. 35.) We define $u\Colon \D\to M$ by
\[u(z):=\left\{\begin{array}{ll}
v\circ \phi(z),&\textrm{if }|z|\leq1/2,\\
f\big(\rho(|z|),z/|z|\big),&\textrm{if }1/2<|z|\leq1.\end{array}\right.\]
This map has the required properties. This proves Claim \ref{claim:u D M}.
\end{proof}
We choose a map $u$ as in Claim \ref{claim:u D M}. By the definition of $\wt\om$ we have $\wt u^*\wt\om=v^*\om-{w'}^*\om_N$. Therefore, equality (\ref{eq:int D u om}) is a consequence of the following: 
\begin{claim}\label{claim:int bar B 1 2} We have 
\begin{equation}\label{eq:int bar B 1 2}\int_{B_{1/2}} u^*\om=\int_\D v^*\om,\quad \int_{\D\wo B_{1/2}}u^*\om=-\int_\D{w'}^*\om_N.\end{equation} 
\end{claim}
\begin{pf}[Proof of Claim \ref{claim:int bar B 1 2}] The first identity in (\ref{eq:int bar B 1 2}) follows from the fact that $\phi$ restricts to a diffeomorphism from $B_{1/2}$ onto $B_1$. To prove the second identity, observe that by the definition of the symplectic form $\om_N$ on the quotient $N_\om$, and (\ref{eq:u D B}), we have on $\D\wo B_{1/2}$, 
\begin{equation}\label{eq:u om u pi N}u^*\om=u^*\pi^*\om_N=(\pi\circ u)^*\om_N=(w'\circ\phi)^*\om_N=\phi^*{w'}^*\om_N.\end{equation}
Since $\phi$ restricts to an orientation reversing diffeomorphism from $B_1\wo \bar B_{1/2}$ onto $B_1\wo \{0\}$, (\ref{eq:u om u pi N}) implies that $\int_{B_1\wo \bar B_{1/2}}u^*\om=-\int_{B_1\wo \{0\}}{w'}^*\om_N$. This implies the second identity in (\ref{eq:int bar B 1 2}).  This proves Claim \ref{claim:int bar B 1 2} and completes the proof of Lemma \ref{le:A M om N}. \end{pf}\end{proof}

\section{Proofs of the main results}\label{sec:proofs}
\begin{proof}[Proof of Theorem \ref{thm:leaf}]\setcounter{claim}{0} Let $M,\om$ and $N$ be as in the hypothesis of this theorem. Without loss of generality we may assume that $N$ is connected. Since $N$ is regular, the symplectic quotient $\big(N_\om,\A_{N,\om|_N},\om_N\big)$ of $(N,\om|_N)$ is well-defined. We define $\wt M,\wt\om$ and $\wt N$ as in (\ref{eq:wt M wt om},\ref{eq:iota N wt N}). Since $N$ is closed, $N_\om$ is closed. By a straight-forward argument the product of two bounded symplectic manifolds is bounded. It follows that $(\wt M,\wt\om)=\big(M\x N_\om,\om\oplus(-\om_N)\big)$ is bounded. Furthermore, by Lemma \ref{le:N}(\ref{le:N:wt N}) $\wt N\sub\wt M$ is a Lagrangian submanifold. It is closed since $N$ is closed. Therefore, applying Theorem \ref{thm:L} with $M,\om$ replaced by $\wt M,\wt\om$, and $L:=\wt N$, there exists a positive constant $C\geq A(\wt M,\wt \om,\wt N)$ such that the statement of that theorem holds. We check that this constant has the required properties: By Lemma \ref{le:A M om N}, we have $C\geq A(M,\om,N)$. Let now $\phi\in\Ham(M,\om)$ be such that $(N,\phi)$ is non-degenerate and inequality (\ref{eq:d phi id M}) is satisfied. We define $\wt\phi$ as in (\ref{eq:wt phi}). Using Lemma \ref{le:phi psi} below, it follows that
\[d^{\wt M,\wt\om}(\wt\phi,\id_{\wt M})\leq d^{M,\om}(\phi,\id)<C.\]
Furthermore, by non-degeneracy of $(N,\phi)$ and Lemma \ref{le:N}(\ref{le:N:non-deg}) we have $\wt\phi(\wt N)\pitchfork\wt N$. Therefore, by the conclusion of Theorem \ref{thm:L}, we have
\begin{equation}\label{eq:wt N cap}\big|\big(\wt N\cap\wt\phi (\wt N)\big)\big|\geq\sum_ib_i(\wt N,\Z_2).\end{equation}
On the other hand, parts (\ref{le:N:phi leaf}) and (\ref{le:N:wt N}) of Lemma \ref{le:N} imply that 
\[\big|\Fix(\phi,N)\big|=\big|\big(\wt N\cap\wt\phi (\wt N)\big)\big|,\quad b_i(\wt N,\Z_2)=b_i(N,\Z_2),\,\forall i.\]
Combining this with (\ref{eq:wt N cap}), inequality (\ref{eq:Fix phi}) follows. This proves Theorem \ref{thm:leaf}.
\end{proof}
\begin{proof}[Proof of Theorem \ref{thm:f}]\setcounter{claim}{0} Let $M,\om,N,f,\iota$ and $\eps$ be as in the hypothesis. We choose a Riemannian metric $g$ on $M$, and denote by $|\cdot|,\ell$ and $d$ the norm on $T_xM$ (for $x\in M$), the length functional, and the distance function, all with respect to $g$. Furthermore, for $x,y\in M$, a linear map $T:T_xM\to T_yM$, and a bilinear map $b:T_xM\x T_xM\to \R$ we define 
\begin{eqnarray*}|T|_\op&:=&\max\big\{|Tv|\,\big|\,v\in T_xM,\,|v|=1\big\},\\
|b|&:=&\max\big\{|b(v,w)|\,\big|\,v,w\in T_xM,\,|v|=|w|=1\big\}.\end{eqnarray*}

Assume first that there exists a closed presymplectic manifold $(M',\om')$ and an $\om'$-horizontal distribution $H\sub TM'$ such that $N$ is the zero-section of $(T{M'}^{\om'})^*$, $M$ is an open neighborhood of $N$, and $\om=\Om_{\om',H}$ (as defined in (\ref{eq:Om om H}) with $\om$ replaced by $\om'$). We identify $M'$ with $N$. We denote by $\pi:(TN^\om)^*\to N$ the canonical projection. We choose a smooth function $\rho:M\to \R$ that has compact support and equals 1 in a neighborhood of $N$. We define $F:=\rho\cdot(f\circ \pi)\Colon M\to \R$. Let $t\in\R$, and $x_0\in \Crit f\sub N$. Then $dF(x_0)=df(x_0)d\pi(x_0)=0$, hence $X_F(x_0)=0$, and therefore $x_0\in \Fix(\phi_F^t,N)$. It follows that $\Crit f\sub \Fix(\phi_F^t,N)$. 
\begin{claim}\label{claim:t 1} There exists a number $t_1>0$ such that for every $t\in(0,t_1]$, we have $\Fix(\phi_F^t,N)\sub\Crit f$. 
\end{claim}
\begin{proof}[Proof of Claim \ref{claim:t 1}] For $(t_0,x_0)\in[0,\infty)\x M$ we abbreviate 
\[\ell(t_0,x_0):=\ell\big([0,t_0]\ni t\mapsto\phi_X^t(x_0)\in M\big).\]
\begin{claim}\label{claim:g} There exists a constant $C$ such that for every $t\in[0,\infty)$, and $x_0\in\Fix(\phi_F^t,N)$, we have
\begin{equation}\label{eq:d g x 0}d(x_0,\phi_F^t(x_0))\leq C\ell(t,x_0)^2.\end{equation} 
\end{claim}
\begin{proof}[Proof of Claim \ref{claim:g}] We denote by $g|_N$ the restriction of $g$ to $TN\oplus TN$. Applying Proposition \ref{prop:X F} below with $\F$ the isotropic foliation of $N$, the horizontal distribution $H$, and $M,g$ replaced by $N,g|_N$, there exists a constant $C$ such that the conclusion of that lemma holds. Let $t\in[0,\infty)$ and $x_0\in\Fix(\phi_F^t,N)$. We define $x:[0,t]\to N$ by $x(s):=\pi\circ\phi_F^s(x_0)$. By Lemma \ref{le:f M} below with $M,\om$ replaced by $N,\om|_N$, we have $\dot x(s)=\pi_*X_F\circ\phi_F^s(x_0)\in H_{x(s)}$, for every $s\in[0,t]$. Furthermore, $x(t)=\phi_F^t(x_0)$ lies in the isotropic leaf through $x(0)$, since $x(0)=x_0\in\Fix(\phi_F^t,N)$. Therefore, the conditions of Proposition \ref{prop:X F} are satisfied, and hence 
\begin{equation}\label{eq:d wt g}d(x_0,\phi_F^t(x_0))\leq d^{g|_N}(x_0,\phi_F^t(x_0))\leq C\ell(x)^2.\end{equation}
Here $d^{g|_N}$ denotes the distance function on $N$ induced by $g|_N$. We denote by $K\sub M$ the support of $\rho$, and $C':=\max_{y\in K}|d\pi(y)|_\op$. Then $\ell(x)\leq C'\ell(t,x_0)$. Combining this with (\ref{eq:d wt g}), we obtain
\[d(x_0,\phi_F^t(x_0))\leq C{C'}^2\ell(t,x_0)^2.\]
Claim \ref{claim:g} follows.
\end{proof}
We choose a constant $C$ as in Claim \ref{claim:g}. We apply Lemma \ref{le:ell f} below with $X:=X_F$, and $f$ replaced by the map $[0,\infty)\ni a\mapsto Ca\in[0,\infty)$, and we choose a constant $t_1:=\eps>0$ as in the assertion of that lemma. Let $t\leq t_1$ and $x_0\in \Fix(\phi_F^t,N)$. Then inequality (\ref{eq:d g x 0}) holds, and therefore by the conclusion of Lemma \ref{le:ell f}, $X_F(x_0)=0$. It follows that $df(x_0)d\pi(x_0)=dF(x_0)=\om(X_F(x_0),\cdot)=0$, and therefore $df(x_0)=0$, i.e.  $x_0\in \Crit f$. This proves Claim \ref{claim:t 1}.
\end{proof}
We choose a number $t_1$ as in Claim \ref{claim:t 1}. Since $F$ has compact support, there exists a number $t_2>0$ so small that $\big\Vert\iota\circ\phi_F^t\circ\iota^{-1}-\id\big\Vert_{C^1(\iota(M))}<\eps$, for every $t\in[0,t_2]$. 
\begin{claim}\label{claim:phi N} There exists $t_3>0$ such that for every $0<t\leq t_3$ the pair $(N,\phi_F^t)$ is non-degenerate. 
\end{claim}
\begin{proof}[Proof of Claim \ref{claim:phi N}] Let $x_0\in \Crit f$. Then $X_F(x_0)=0$, and hence the derivative $dX_F(x_0):T_{x_0}M\to T_{x_0}M$ is well-defined. We fix $t\in\R$. Then $\phi_F^t(x_0)=x_0$, and hence $d\phi_F^t(x_0)$ is a linear map from $T_{x_0}M$ to $T_{x_0}M$. Hence we may define 
\begin{equation}\label{eq:T x 0}T_{x_0}^t:=d\phi_F^t(x_0)-\id-tdX_F(x_0):T_{x_0}M\to T_{x_0}M.\end{equation}
We have $T_{x_0}^0=0$. Furthermore, by a calculation in local coordinates, we have $\left.\frac d{dt}\right|_{t=0}d\phi_F^t(x_0)=dX_F(x_0)$. Hence by Taylor's theorem there exists a constant $C_{x_0}$ such that $|T_{x_0}^t|_\op\leq C_{x_0}t^2,$ for every $t\in[0,t_1]$. 

A calculation in Darboux charts shows that the bilinear form $B_{x_0}:T_{x_0}M\x T_{x_0}M\ni (v,w)\mapsto \om\big(dX_F(x_0)v,w\big)\in\R$ is the Hessian of $F$. Since $F|_N=f$, it follows that the restriction $b_{x_0}:=B_{x_0}|_{T_{x_0}N\x T_{x_0}N}$ is the Hessian of $f$. We define the linear map $A:T_{x_0}N\to T_{x_0}N$ by $g_{x_0}(\cdot, A\cdot):=b_{x_0}$, and we denote by $V_+$ and $V_-$ the direct sum of the positive and negative eigenspaces of $A$, respectively. It follows that $A$ is self-adjoint with respect to $g_{x_0}$. Since by assumption $f$ is Morse, the form $b_{x_0}$ is non-degenerate, hence $A$ is an isomorphism, and therefore $T_{x_0}N=V_+\oplus V_-$. We define $c_{x_0}:=\min\{|\lam|\,|\,\lam\textrm{ eigenvalue of }A\big\}$. Since $f$ is Morse the set $\Crit f$ is isolated. Since $N$ is compact, it follows that $\Crit f$ is finite. Hence we may define 
\begin{equation}\label{eq:t 3}t_3:=\min\left(\left\{\frac{c_{x_0}}{2C_{x_0}|\om_{x_0}|}\,\Big|\,x_0\in \Crit f\right\}\cup\{t_1\}\right).\end{equation}
For $x_0\in N$ we denote by $\pr_{x_0}:T_{x_0}N\to T_{x_0}N/T_{x_0}N^\om$ the canonical projection. Let $0<t\leq t_3$, $F\sub N$ be a leaf, $x\in C^\infty([0,1],F)$ a path, and $v\in T_{x(0)}N\cap T_{x(0)}(\phi_F^t)^{-1}(N)$. Assume that $\phi_F^t(x(0))=x(1)$, and that 
\begin{equation}\label{eq:hol =}\hol^{\om,N}_x\pr_{x(0)}v=\pr_{x(1)}d\phi_F^t(x(0))v.\end{equation}
Claim \ref{claim:phi N} is a consequence of the following. 
\begin{claim}\label{claim:v=0} We have $v=0$. 
\end{claim}
\begin{pf}[Proof of Claim \ref{claim:v=0}] Since $\phi_F^t(x_0)=x(1)\in F$, we have $x_0\in\Fix(\phi_F^t,N)$. Therefore, using $t\leq t_3\leq t_1$, Claim \ref{claim:t 1} implies that $x_0\in\Crit f$, and hence $x(1)=x(0)$. Recall that $\pi_N$ denotes canonical projection from $N$ to the set of isotropic leaves $N_\om$. We abbreviate $x_0:=x(0)$. Since by assumption $N$ is regular, it follows from (\ref{eq:hol =}) and Lemma \ref{le:F x} that $d\pi_N(x_0)d\phi_F^t(x_0)v=d\pi_N(x_0)v$. By Lemma \ref{le:X/R}(\ref{le:X/R:ker}) this means that 
\begin{equation}\label{eq:d phi F t}d\phi_F^t(x_0)v-v\in \ker d\pi(x_0)=T_{x_0}N^\om.\end{equation}
We define $v_\pm\in V_\pm$ by $v_++v_-:=v$. Since $A$ is $g_{x_0}$-self-adjoint, eigenvectors of $A$ for distinct eigenvalues are $g_{x_0}$-orthogonal to each other. It follows that $b_{x_0}(v_-,v_+)=0$, and therefore, 
\begin{equation}\label{eq:big om}c_{x_0}|v_\pm|^2\leq\big|b_{x_0}(v,v_\pm)\big|=\big|\om\big(dX_F(x_0)v,v_\pm\big)\big|.\end{equation}
By (\ref{eq:T x 0}) and (\ref{eq:d phi F t}) we have
\begin{equation}\label{eq:0}t\om\big(dX_F(x_0)v,v_\pm\big)+\om(T_{x_0}^tv,v_\pm)=\om\big(d\phi_F^t(x_0)v-v,v_\pm\big)=0.\end{equation}
Furthermore, we may estimate
\begin{equation}\label{eq:big om T}\big|\om(T_{x_0}^tv,v_\pm)\big|\leq |\om_{x_0}||T_{x_0}^t|_\op|v||v_\pm|.\end{equation}
Consider the case $|v_+|\geq|v_-|$. Then $|v|\leq\sqrt2|v_+|$, since $V_+$ and $V_-$ are orthogonal with respect to $g_{x_0}$. Hence (\ref{eq:big om},\ref{eq:0},\ref{eq:big om T}) imply that $c_{x_0}|v_+|^2t\leq\sqrt 2|\om_{x_0}||T_{x_0}^t|_\op|v_+|^2$. Combining this with the inequalities $|T_{x_0}^t|_\op\leq C_{x_0}t^2$ and $t\leq t_3$, and (\ref{eq:t 3}), we obtain $c_{x_0}|v_+|^2t\leq(1/\sqrt2)c_{x_0}|v_+|^2t$. Since $c_{x_0},t>0$, it follows that $0=|v_+|\geq|v_-|$, and therefore $v=v_++v_-=0$. The case $|v_-|\geq|v_+|$ is treated in an analogous way. This proves Claim \ref{claim:v=0} and completes the proof of Claim \ref{claim:phi N}.
\end{pf}\end{proof}
We choose a number $t_3$ as in Claim \ref{claim:phi N}, and define $\phi:=\phi_F^{\min\{t_1,t_2,t_3\}}$. This map has the required properties. This proves Theorem \ref{thm:f} in the case in which $M$ and $N$ are an open neighborhood and the zero-section of $(T{M'}^{\om'})^*$, for some closed presymplectic manifold $(M',\om')$, and $\om$ is of the form $\Om_{\om',H}$. 

Consider now the general case. Let $M,\om,N,f,\iota$ and $\eps$ be as in the hypothesis of Theorem \ref{thm:f}. We denote by $\wt N\sub (TN^\om)^*$ the zero section. We choose an $\om$-horizontal distribution $H\sub TN$, and define $\wt\om:=\Om_{\om|_N,H}$. By a theorem by C.-M.~Marle there exist compact neighborhoods $K\sub M$ of $N$ and $\wt K\sub (TN^\om)^*$ of $\wt N$, and a diffeomorphism $\psi:K\to \wt K$ such that $\psi(N)=\wt N$ and $\psi^*\wt\om=\om$. (See 4.5. Th\'eor\`eme on p.~79 in \cite{Ma}.) We define 
\[\wt f:=f\circ\psi|_N^{-1}:\wt N\to \R,\quad \wt\iota:=\iota\circ\psi^{-1}:\wt K\to \R^{2n}.\]
We denote by $\Int\wt K$ the interior of $\wt K$. We proved that there exists $\wt\phi\in \Ham_\c(\Int\wt K,\wt\om)$ such that $\Fix(\wt\phi,\wt N)=\Crit \wt f$, $\big(\wt N,\wt\phi,\wt\om\big)$ is non-degenerate, and $\big\Vert \wt\iota\circ\wt\phi\circ\wt\iota^{-1}-\id\Vert_{C^1(\wt\iota(\Int\wt K))}<\eps$. We define $\phi:M\to M$ to be the extension of $\psi^{-1}\circ\wt\phi\circ\psi:\Int K\to\Int K$ by the identity. It follows that $\phi\in\Ham_\c(M,\om)$,  
\[\Fix(\phi,N)=\psi^{-1}(\Fix(\wt\phi,\wt N))=\Crit (\wt f\circ\psi|_N)=\Crit f, \]
and $(N,\phi,\om)$ is non-degenerate. Furthermore, $\iota\circ\phi\circ\iota^{-1}=\wt\iota\circ\wt\phi\circ\wt\iota^{-1}$ on $\iota(K)=\wt\iota(\wt K)\sub\R^{4n}$, and therefore, 
\[\big\Vert\iota\circ\phi\circ\iota^{-1}-\id\big\Vert_{C^1(\iota(M))}=\big\Vert\wt\iota\circ\wt\phi\circ\wt\iota^{-1}-\id\big\Vert_{C^1(\wt \iota(\wt K))}<\eps.\]
Hence $\phi$ satisfies the required properties. This proves Theorem \ref{thm:f} in the general case.
\end{proof}
\begin{proof}[Proof of Proposition \ref{prop:Stiefel}] \setcounter{claim}{0} Note that the isotropic leaf through a point $\Th\in V(k,n)$ is the orbit $\U(k)\cdot\Th$ of the action of $U(k)$ on $V(k,n)$ by multiplication from the left. 
\begin{claim}\label{claim:A leq pi} We have $A\big(\C^{k\x n},\om_0,V(k,n)\big)\leq\pi$.
\end{claim}
\begin{proof}[Proof of Claim \ref{claim:A leq pi}] Consider the map $u:\D\to \C^{k\x n}$ defined by $u^1_{\phantom{1}1}(z):=z$, for $z\in\D$, $u^i_{\phantom{i}i}\const1$, for $i=2,\ldots,k$, and $u^i_{\phantom{i}j}\const0$, for $i=1,\ldots,k$, $j=1,\ldots,n$ such that $i\neq j$. Then $u(S^1)$ is contained in the leaf $\U(k)\cdot u(1)$. Furthermore, $\int_\D u^*\om_0=\pi$. Hence $\pi\in S\big(\C^{k\x n},\om_0,V(k,n)\big)$ (defined as in (\ref{eq:S})). Claim \ref{claim:A leq pi} follows from this. 
\end{proof}
\begin{claim}\label{claim:A geq pi}We have $A\big(\C^{k\x n},\om_0,V(k,n)\big)\geq\pi$.
\end{claim}
\begin{proof}[Proof of Claim \ref{claim:A geq pi}] Let $u\in C^\infty(\D,\C^{k\x n})$, and assume that $\int_\D u^*\om_0>0$, and that there exists an isotropic leaf $F\sub V(k,n)$ such that $u(z)\in F$, for every $z\in S^1$. Since the action of $U(k)$ on $V(k,n)$ is free, there exists a unique map $g_0:S^1\to U(k)$ such that $u(z)=g_0(z)u(1)$, for $z\in S^1$. This map is smooth. We define $d$ to be the degree of the map $\det\circ g_0:S^1\to S^1$. Claim \ref{claim:A geq pi} is a consequence of the following.
\begin{claim}\label{claim:int D u} We have $\int_\D u^*\om_0=d\pi$. 
\end{claim}
\begin{pf}[Proof of Claim \ref{claim:int D u}] We define 
\[h_0:S^1\to U(k),\quad h_0(z):=\diag(z^d,1,\ldots,1)g_0(z)^{-1}.\]
Here $\diag(a_1,\ldots,a_k)$ means the diagonal $k\x k$-matrix with diagonal entries $a_1,\ldots,a_k$. The map $\det\circ h_0:S^1\to S^1$ has degree 0. Since the determinant induces an isomorphism of the fundamental groups of $U(k)$ and $S^1$, it follows that there exists a continuous homotopy from the constant map $\one$ to $h_0$. This gives rise to a continuous map $h:\D\to U(k)$ such that $h|_{S^1}=h_0$. We may assume \Wlog that $h$ is smooth. Let $\mu:\C^{k\x n}\to\Lie\U(k)$ be a moment map for the action of $U(k)$ on $V(k,n)$. By a straight-forward calculation, we have
\[(hu)^*\om_0=u^*\om_0-d\big\lan\mu\circ u,h^{-1}dh\big\ran,\]
see Lemma 9 in \cite{Zi2}. By Stokes' theorem, it follows that 
\begin{equation}\label{eq:int u om int}\int u^*\om_0=\int_\D (hu)^*\om_0.\end{equation} 
We define the map $v:\D\to\C^{k\x n}$ by $v(z):=\diag(z^d,1,\ldots,1)u(1)$, for $z\in\D$. Then for $z\in S^1$, we have $(hu)(z)=(h_0u)(z)=v(z)$. Therefore, 
\[\int_\D (hu)^*\om_0=\int_\D v^*\om_0=d\pi\sum_{j=1,\ldots,n}\big|u^1_{\phantom{1}j}(1)\big|^2=d\pi.\]
Here in the last equality we used that the first row of $u(1)$ has norm 1. This proves Claim \ref{claim:int D u} and hence Claim \ref{claim:A geq pi}.
\end{pf}\end{proof}
Claims \ref{claim:A leq pi} and \ref{claim:A geq pi} imply that $A\big(\C^{k\x n},\om_0,V(k,n)\big)=\pi$. This proves the first assertion. To prove the second assertion, let $C>\pi$. Then there exists $\wt H\in C^\infty_\c(\C,\R)$ such that $\phi_{\wt H}^1(\D)\cap\D=\empty$, and $\Vert\wt H\Vert<C$. (See for example the proof of Proposition 1.4. in \cite{Ho}.) We define $\pi:\C^{k\x n}\to \C$ by $\pi(\Th):=\Th^1_{\phantom{1}1}$, and choose $\rho\in C^\infty(\C^{k\x n},[0,1])$ with compact support, such that $\rho=1$ on $\bigcup_{t\in[0,1]}\phi_{\wt H\circ \pi}^t(V(k,n))$. We define $H:=\rho\cdot(\wt H\circ \pi):\C^{k\x n}\to\R$ and $\phi:=\phi_H^1$. Then $\phi\in\Ham_\c(\C^{k\x n},\om_0)$, and $d_\c(\phi,\id)\leq\Vert H\Vert=\Vert\wt H\Vert<C$. Furthermore, $\pi(V(k,n))=\D$ and $\pi\circ\phi_{\wt H\circ\pi}^1=\phi_{\wt H}^1\circ\pi$. It follows that 
\[\pi(V(k,n))\cap\pi\circ\phi_{\wt H\circ\pi}^1(V(k,n))\sub\D\cap\phi_{\wt H}^1(\D)=\empty.\] 
Hence $V(k,n)\cap \phi_{\wt H\circ\pi}^1(V(k,n))=\empty$. Our choice of $\rho$ implies that $\phi_{\wt H\circ\pi}^1=\phi_H^1$ on $V(k,n)$. It follows that $V(k,n)\cap\phi_H^1(V(k,n))=\empty$. Therefore, $\phi$ has the required properties. This proves Proposition \ref{prop:Stiefel}.
\end{proof}
\begin{proof}[Proof of Proposition \ref{prop:Arnold}]\setcounter{claim}{0} Let $M_i,\om_i,L,M,\om,N$ and $\psi$ be as in the hypothesis of the Conjecture. Without loss of generality we may assume that $M_1$ and $L$ are connected. We define
\begin{eqnarray*}&\wt M:=M_1\x M_2\x M_1,\quad \wt\om:=\om_1\oplus\om_2\oplus(-\om_1),&\\
&\wt N:=\big\{(x_1,x_2,x_1)\in M\,\big|\,x_1\in M_1,\,x_2\in L\big\},\quad \wt\phi:=\phi\x\id_{M_1}.&\end{eqnarray*}
Then $\wt N$ is a Lagrangian submanifold of $\wt M$. Since $M_1$ and $L$ are closed manifolds by assumption, it follows that $\wt N$ is closed. The map 
\[\wt \psi\Colon \wt M\to \wt M,\quad \wt\psi(x_1,x_2,y):=(y,\psi(x_2),x_1).\]
is an $\wt\om$-anti-symplectic involution whose fixed point set equals $\wt N$. Furthermore, $N=M_1\x L\sub M=M_1\x M_2$ is a regular $\om$-coisotropic submanifold, and the symplectic quotient $\big(N_\om,\A_{N,\om|_N},\om_N\big)$ of $(N,\om|_N)$ is isomorphic to $(M_1,\om_1)$ via the map $M_1\ni y\mapsto \{y\}\x L\in N_\om$. Via this map, the definitions of $\wt M,\wt\om$ and $\wt \phi$ agree with (\ref{eq:wt M wt om},\ref{eq:iota N wt N},\ref{eq:wt phi}). Hence by non-degeneracy of $(N,\phi)$ and Lemma \ref{le:N}(\ref{le:N:non-deg}), we have $\wt N\pitchfork \wt\phi(\wt N)$. Thus $\wt M,\wt\om$ and $\wt N$ satisfy the hypotheses of the Lagrangian AGC. Supposing that this conjecture is true, it follows that 
\begin{equation}\label{eq:wt N cap wt phi}\big|\wt N\cap\wt\phi(\wt N)\big|\geq \sum_ib_i(\wt N,\Z_2).\end{equation}
The manifolds $N$ and $\wt N$ are diffeomorphic, and therefore their Betti sums agree. Furthermore, Lemma \ref{le:N}(\ref{le:N:phi leaf}) implies that $\big|\Fix(\phi,N)\big|=\big|\wt N\cap\wt\phi(\wt N)\big|$. Combining this with (\ref{eq:wt N cap wt phi}), we obtain inequality (\ref{eq:Fix phi}). This proves Proposition \ref{prop:Arnold}.
\end{proof}
\begin{proof}[Proof of Corollary \ref{cor:presympl}]\setcounter{claim}{0} Let $(M,\om)$ be a bounded and aspherical symplectic manifold, and $(M',\om')$ a closed and regular presymplectic manifold of corank $\dim M-\dim M'$. Assume that $M'$ has a simply-connected isotropic leaf $F_0$, and that there exists an embedding $\psi\Colon M'\to M$ satisfying $\psi^*\om=\om'$. Let $\phi\in\Ham(M,\om)$. It suffices to prove that $N:=\psi(M')$ intersects $\phi(N)$. Replacing $M'$ by the connected component of $M'$ containing $F_0$, we may assume without loss of generality that $M'$ is connected. It follows from Proposition \ref{prop:M om} below that the submanifold $\psi(M')\sub M$ is coisotropic. We choose a constant $C>0$ as in Theorem \ref{thm:leaf}. If $(N,\phi)$ is degenerate, then by definition $\Fix(\phi,N)\neq\emptyset$, and hence $N\cap\phi(N)\neq\emptyset$. So assume that $(N,\phi)$ is non-degenerate. 
\begin{claim}\label{claim:A M om N} We have $A(M,\om,N)=\infty$. 
\end{claim}
\begin{proof}[Proof of Claim \ref{claim:A M om N}] Let $u\in C^\infty(\D,M)$ be a smooth map such that there exists $F\in N_\om$ satisfying $u(S^1)\sub F$. Since $N$ is closed and the canonical projection $\pi_N\Colon N\to N_\om$ is a submersion, by the proposition on p. 31 in \cite{Eh} it is a locally trivial fiber bundle. Since $N$ is connected, it follows that $F$ is diffeomorphic to $F_0$, and therefore simply connected. Hence there exists a smooth map $v\Colon \D\to F$ such that $u$ and $v$ agree on the boundary $S^1$. We choose a map $\rho\in C^\infty([0,1],[0,1])$ such that $\rho(r)=r$ for $r\leq1/2$, $\rho(1)=1$, $\rho'(r)>0$, for every $r\in(0,1)$, and all derivatives of $\rho$ vanish at $r=1$. We define $f\Colon \D\to \D$ by $f(rz):=\rho(r)z$, for $r\in[0,1]$ and $z\in S^1$. We denote by $\BAR{\D}$ the disk with the reversed orientation, and by $w:=(u\circ f)\# (v\circ f)\Colon S^2\iso\D\#\BAR{\D}\to M$ the concatenation of the maps $u\circ f$ and $v\circ f$. This map is smooth, and since $(M,\om)$ is symplectically aspherical, we have
\begin{equation}\label{eq:0 int S 2}0=\int_{S^2}w^*\om=\int_{B_1}(u\circ f)^*\om-\int_\D(v\circ f)^*\om.\end{equation}
Since $v$ takes values in the isotropic leaf $F$, we have $(v\circ f)^*\om=0$. Thus (\ref{eq:0 int S 2}) implies that $\int_\D u^*\om=\int_{B_1}(u\circ f)^*\om=0$. Claim \ref{claim:A M om N} follows from this.
\end{proof}
Claim \ref{claim:A M om N} implies that $C=\infty$, and hence inequality (\ref{eq:d phi id M}) holds. Therefore, the conditions of Theorem \ref{thm:leaf} are satisfied. Applying this theorem, inequality (\ref{eq:Fix phi}) follows, and therefore $N\cap\phi(N)\neq\emptyset$. This proves Corollary \ref{cor:presympl}.
\end{proof}
\appendix

\section{Auxiliary results}\label{sec:aux}
\subsection{(Pre-)symplectic geometry}\label{subsec:sympl geo}
The following result was used in Section \ref{sec:main} and in the proof of Corollary \ref{cor:presympl}. It will also be needed for the proof of Proposition \ref{prop:N} below.
\begin{prop}\label{prop:M om} Let $(M,\om)$ be a presymplectic manifold, $M'$ a manifold, $\psi\Colon M'\to M$ an immersion, and $\om':=\psi^*\om$. If $\om'$ has constant corank then the inequality 
\begin{equation}\label{eq:dim M' corank om'}\dim M'+\corank\om'\leq\dim M+\corank\om
\end{equation}
holds. Suppose now that $\psi$ is an embedding. Then $\psi(M')\sub M$ is coisotropic if and only if $\om'$ has constant corank and equality in (\ref{eq:dim M' corank om'}) holds. 
\end{prop}
For the proof of this proposition we need the following lemma. 
\begin{lemma}\label{le:V om V' om'} Let $(V,\om)$ and $(V',\om')$ be presymplectic vector spaces (possibly $\infty$-dimensional), and $\Psi\Colon V'\to V$ a linear map such that $\Psi^*\om=\om'$. Then
\begin{equation}\label{eq:dim V' corank om'}\dim V'+\corank\om'\leq \dim V+\corank\om+\dim \ker \Psi+\dim\ker\Psi|_{{V'}^{\om'}}.\end{equation}
Furthermore, if $\dim V,\dim V'<\infty$ then $\Psi V'\sub V$ is coisotropic if and only if equality in (\ref{eq:dim V' corank om'}) holds.
\end{lemma}
The proof of this lemma is based on the following. 
\begin{lemma}\label{le:dim W om} Let $(V,\om)$ be a presymplectic vector space, and $W\sub V$ a subspace. Then
\begin{equation}\label{eq:dim W om}\dim W+\dim W^\om\leq\dim V+\dim V^\om.
\end{equation}
Furthermore, if $\dim V<\infty$ and $V^\om\sub W$ then equality in (\ref{eq:dim W om}) holds.
\end{lemma}
\begin{proof}[Proof of Lemma \ref{le:dim W om}]\setcounter{claim}{0} To see that (\ref{eq:dim W om}) holds, we define the linear map $\om^\#\Colon V\to V^*$ by $\om^\#v:=\om(v,\cdot)$. We denote by $i_W\Colon W\to V$ the inclusion. Then $W^\om=\ker(i_W^*\om^\#)$, and therefore,
\begin{equation}\label{eq:dim V}\dim\im(i_W^*\om^\#)+\dim W^\om=\dim V.\end{equation}
Consider the canonical isomorphism $\iota\Colon V\to V^{**}$, $\iota(v)(\phi):=\phi(v)$. A direct calculation shows that $(\om^\#)^*\iota=-\om^\#$. It follows that $(\om^\# i_W)^*\iota=-i_W^*\om^\#$, and therefore
\begin{equation}\label{eq:dim i W}\dim\im(\om^\# i_W)=\dim\im(\om^\# i_W)^*=\dim\im(i_W^*\om^\#).\end{equation}
On the other hand, we have $\dim\ker(\om^\#i_W)\leq\dim\ker(\om^\#)=\dim V^\om$. Combining this with (\ref{eq:dim i W}), we obtain
\[\dim W=\dim\ker(\om^\#i_W)+\dim\im(\om^\# i_W)\leq\dim V^\om+\dim\im(i_W^*\om^\#).\]
This together with (\ref{eq:dim V}) implies (\ref{eq:dim W om}).

Assume now that $V^\om\sub W$. Then $\dim\ker(\om^\#i_W)=\dim\ker(\om^\#)$, and therefore the above argument shows that equality in (\ref{eq:dim W om}) holds. This proves Lemma \ref{le:dim W om}.
\end{proof}
\begin{proof}[Proof of Lemma \ref{le:V om V' om'}]\setcounter{claim}{0} The hypothesis $\Psi^*\om=\om'$ implies that 
\begin{equation}\label{eq:Psi V'} \Psi({V'}^{\om'})\sub (\Psi V')^\om.\end{equation}
It follows that 
\begin{equation}\label{eq:dim V' om'}\dim {V'}^{\om'}=\dim\Psi({V'}^{\om'})+\dim\ker\Psi|_{{V'}^{\om'}}\leq\dim(\Psi V')^\om+\dim\ker\Psi|_{{V'}^{\om'}},\end{equation}
and therefore 
\begin{eqnarray}\nn\dim V'+\corank\om'&=&\dim(\Psi V')+\dim\ker\Psi+\dim{V'}^{\om'}\\
\nn&\leq&\dim(\Psi V')+\dim(\Psi V')^\om\\
\label{eq:dim Psi V' Psi V' om}&&+\dim\ker\Psi+\dim\ker\Psi|_{{V'}^{\om'}}.
\end{eqnarray}
Applying Lemma \ref{le:dim W om}, inequality (\ref{eq:dim V' corank om'}) follows. The second statement is a consequence of the following two claims.
\begin{claim}\label{claim:coiso Psi V'} $\Psi V'$ is coisotropic if and only if equality in (\ref{eq:Psi V'}) holds.
\end{claim}
\begin{proof}[Proof of Claim \ref{claim:coiso Psi V'}] If equality in (\ref{eq:Psi V'}) holds then $(\Psi V')^\om=\Psi({V'}^{\om'})\sub \Psi V'$, hence $\Psi V'$ is coisotropic. Conversely, if $\Psi V'$ is coisotropic then a straight-forward argument implies that $(\Psi V')^\om\sub \Psi({V'}^{\om'})$, hence equality in (\ref{eq:Psi V'}) holds. This proves Claim \ref{claim:coiso Psi V'}.
\end{proof}
Assume now that $\dim V,\dim V'<\infty$.
\begin{claim}\label{claim:Psi V' dim V'} Equality in (\ref{eq:Psi V'}) holds if and only if equality in (\ref{eq:dim V' corank om'}) holds.
\end{claim}
\begin{pf}[Proof of Claim \ref{claim:Psi V' dim V'}] Suppose that equality in (\ref{eq:Psi V'}) holds. Then equality in (\ref{eq:dim V' om'}) and in (\ref{eq:dim Psi V' Psi V' om}) holds. Furthermore, $V^\om\sub(\Psi V')^\om=\Psi ({V'}^{\om'})\sub\Psi V'$, and hence by Lemma \ref{le:dim W om}
\[\dim(\Psi V')+\dim(\Psi V')^\om=\dim V+\corank\om.\]
Combining this with (\ref{eq:dim Psi V' Psi V' om}), it follows that equality in (\ref{eq:dim V' corank om'}) holds. 

Assume now on the contrary that equality in (\ref{eq:dim V' corank om'}) holds. Then
\begin{eqnarray}\nn \dim(\Psi V')^\om&\leq&\dim V+\dim V^\om-\dim\Psi V'\\
\nn&=&\dim V'-\dim\Psi V'-\dim\ker\Psi+\dim {V'}^{\om'}-\dim\ker\Psi|_{{V'}^{\om'}}\\
\nn&=&0+\dim\Psi({V'}^{\om'}).
\end{eqnarray}
Here in the first step we used Lemma \ref{le:dim W om}. It follows that equality in (\ref{eq:Psi V'}) holds. This proves Claim \ref{claim:Psi V' dim V'} and concludes the proof of Lemma \ref{le:V om V' om'}.
\end{pf}
\end{proof}
\begin{proof}[Proof of Proposition \ref{prop:M om}]\setcounter{claim}{0} Let $M,\om,M',\psi$ and $\om'$ be as in the hypothesis. We choose a point $x'\in M'$, and define 
\[V':=T_{x'}M',\,\Om':=\om'|_{V'\x V'},\,V:=T_{\psi(x')}M,\,\Om:=\om|_{V\x V},\,\Psi:=d\psi(x').\]
Then the hypotheses of Lemma \ref{le:V om V' om'} are satisfied with $\om,\om'$ replaced by $\Om,\Om'$. It follows that inequality (\ref{eq:dim V' corank om'}) holds. Since $\Psi$ is injective, this implies inequality (\ref{eq:dim M' corank om'}), provided that $\om'$ has constant corank.

Suppose now that $\psi$ is an embedding. Assume that $\om'$ has constant corank and equality in (\ref{eq:dim M' corank om'}) holds. Let $x'\in M'$. Applying Lemma \ref{le:V om V' om'} with $V,V'$ and $\Psi$ as above, and $\om,\om'$ replaced by $\Om,\Om'$, it follows that $T_{\psi(x')}\psi(M')=\Psi V'\sub T_{\psi(x')}M$ is coisotropic. Hence $\psi(M')\sub M$ is coisotropic. Conversely, assuming that $\psi(M')\sub M$ is coisotropic, Lemma \ref{le:V om V' om'} implies that the corank of $\om'$ at any point $x'\in M'$ equals $\dim M+\corank\om-\dim M'$. This proves Proposition \ref{prop:M om}.
\end{proof}
The next result was used in Section \ref{sec:main}. Let $(X,\si)$ be a closed symplectic manifold, $\pi\Colon E\to X$ a closed smooth fiber bundle, $H\sub TE$ a horizontal subbundle, and let $N,\pi_X,\iota_H,\om:=\Om_{\si,H}$ and $M$ be as in the construction explained in that section, on p.~\pageref{V}.
\begin{prop}\label{prop:N} $N\sub M$ is a regular coisotropic submanifold. Furthermore, if $(X,\si)$ is symplectically aspherical then $A(M,\om,N)=\infty$.
\end{prop}
For the proof of Proposition \ref{prop:N} we need the following. We denote by $i_E$ the embedding of $E$ as the zero section $N\sub V^*E$. 
\begin{lemma}\label{le:pi si} We have $\pi^*\si=i_E^*\om$. 
\end{lemma}
\begin{proof}[Proof of Lemma \ref{le:pi si}]\setcounter{claim}{0} We denote by $j_E$ the embedding of $E$ as the zero-section of $T^*E$. Then $\iota_H\circ i_E=j_E$, and therefore, denoting by $\lam_\can$ the canonical one-form on $T^*E$, we obtain $i_E^*\iota_H^*\om_\can=-dj_E^*\lam_\can=-d0=0$. Since $\pi_X\circ i_E=\pi$, it follows that $\pi^*\si=i_E^*\big(\pi_X^*\si+\iota_H^*\om_\can\big)=i_E^*\om$. This proves Lemma \ref{le:pi si}. 
\end{proof}
For any manifold $M$ and any positive integer $k$ we denote by $\Om^k(M)$ the space of differential forms of degree $k$. 
\begin{lemma}\label{le:homotopy} Let $M$ and $N$ be smooth manifolds, $k\geq1$, $\om\in\Om^k(N)$ a closed form, and $u\Colon [0,1]\x M\to N$ a smooth map such that $u(t,x)=u(0,x)$, for every $t\in[0,1]$, $x\in \d M$. Then there exists $\al\in\Om^{k-1}(M)$ such that $d\al=u(1,\cdot)^*\om-u(0,\cdot)^*\om$, and $\al(x)=0$, for all $x\in\d M$.
\end{lemma}
\begin{proof}[Proof of Lemma \ref{le:homotopy}]\setcounter{claim}{0} This follows from the proof of Theorem VI(7.13) p. 270 in the book \cite{Bo}. 
\end{proof}
\begin{rmk}\label{rmk:leaf} Let $(M,\om)$ and $(M',\om')$ be presymplectic manifolds, and $\psi\Colon M'\to M$ a diffeomorphism such that $\psi^*\om=\om'$. Then by a straight-forward argument the image of every isotropic leaf of $M'$ under $\psi$ is an isotropic leaf of $M$. Furthermore, the map $M'_{\om'}\ni F'\mapsto \psi(F')\in M_\om$ is a bijection. 
\end{rmk}
\begin{proof}[Proof of Proposition \ref{prop:N}]\setcounter{claim}{0} We may assume without loss of generality that $E$ is connected. By Lemma \ref{le:pi si} we have $i_E^*\om=\pi^*\si$. Furthermore, for every $e\in E$ we have $T_eE^{\pi^*\si}=\ker d\pi(e)$, and hence $i_E^*\om$ has constant corank equal to the dimension of the fiber of $E$. It follows that equality in (\ref{eq:dim M' corank om'}) holds with $M':=E$, $\psi:=i_E$, and $\om':=\psi^*\om$. Hence by Proposition \ref{prop:M om} the submanifold $N\sub M$ is coisotropic. Furthermore, the leaf relation of $(E,\pi^*\si)$ consists of all pairs $(x_0',x_1')\in E\x E$ that lie in the same connected component of one of the fibers of $E$. It follows from an argument involving local trivializations for $E$ that this is a closed subset and a submanifold. Hence $(E,\pi^*\si)$ is regular. Since $\om|_N$ is the push-forward of $\pi^*\si$ under the diffeomorphism $i_E\Colon E\to N$, it follows that $N$ is regular. 

To prove the second statement, assume that $(X,\si)$ is aspherical. Let $u\in C^\infty(\D,M)$ be a map such that there exists a leaf $F\in N_\om$ satisfying $u(S^1)\sub F$. It suffices to prove that $\int_\D u^*\om=0$. To see this, we denote by $\pi_0$ the canonical projection from $V^*E$ to its zero-section $N$. We choose a smooth function $\rho\Colon [0,1]\to [0,1]$ such that $\rho(r)=r$, for $r\leq1/3$, and $\rho(r)=1$, for $r\geq2/3$. We define $u_0\Colon \D\to M$ by $u_0(rz):=\pi_0\circ u(\rho(r)z)$, for $r\in[0,1]$ and $z\in S^1$.
\begin{claim}\label{claim:int D u u 0} We have
\begin{equation}\label{eq:int D u u 0}\int_\D u^*\om=\int_\D u_0^*\om.\end{equation}
\end{claim}
\begin{proof}[Proof of Claim \ref{claim:int D u u 0}] We define the map $h\Colon [0,1]\x V^*E\to V^*E$ by $h(t,e,\al):=(e,t\al)$, and the map $f\Colon [0,1]\x\D\to V^*E$ by
\[f(t,rz):=h\Big(t,u\big(\big(tr+(1-t)\rho(r)\big)z\big)\Big),\,\forall r\in[0,1],\,z\in S^1.\] 
Observe that $f(0,\cdot)=u_0$, and $f(1,\cdot)=u$. Since $u(S^1)\sub F\sub N$, we have $f(t,z)=f(0,z)$, for every $t\in[0,1]$ and $z\in S^1$. Hence the hypotheses of Lemma \ref{le:homotopy} are satisfied, with $M,N$ replaced by $\D,V^*E$, and $u$ replaced by the map $f$. If follows that there exists $\al\in\Om^1(\D)$ such that 
\[d\al=u^*\om-u_0^*\om,\quad \al(z)=0,\,\forall z\in S^1.\]
Together with Stokes' Theorem this implies (\ref{eq:int D u u 0}). This proves Claim \ref{claim:int D u u 0}.
\end{proof}
\begin{claim}\label{claim:int D u 0} We have $\int_\D u_0^*\om=0$. \end{claim}
\begin{proof}[Proof of Claim \ref{claim:int D u 0}] To see this, we choose an orientation preserving diffeomorphism $\phi\Colon \C\to B_1$, and we define the map $f\Colon S^2\iso\C\cup\{\infty\}\to X$ by 
\[f(z):=\left\{\begin{array}{ll}
\pi_X\circ u_0\circ\phi(z),&\textrm{if }z\in \C,\\
\pi_X\circ u(1),&\textrm{if }z=\infty.
\end{array}\right.\]
\begin{claim}\label{claim:f} This map is smooth.
\end{claim}
\begin{proof}[Proof of Claim \ref{claim:f}] $f|_\C$ is smooth. Furthermore, by Remark \ref{rmk:leaf} there exists a $\pi^*\si$-isotropic leaf $F'$ of $E$ such that $i_E(F')=F$. Let $e_0\in F'$. Since for every $e\in E$ we have $T_eE^{\pi^*\si}=\ker d\pi(e)$, it follows that $F'$ is the connected component of the fiber of $E$ containing $e_0$. Since $\pi_X\circ i_E=\pi$, this implies that $\pi_X$ equals the constant $\pi(e_0)\in X$ on $F$. We choose a number $r_0>0$ such that $|\phi(rz)|\geq2/3$, for $r\geq r_0$ and $z\in S^1$. Let $z\in \C\wo B_{r_0}$. Then $u_0\circ \phi(z)\in u(S^1)\sub F$, and therefore $f(z)=\pi(e_0)$. Since also $f(\infty)=\pi(e_0)$, it follows that $f$ is smooth on $S^2$. This proves Claim \ref{claim:f}.
\end{proof}
By Claim \ref{claim:f} and symplectic asphericity of $X$ we have
\begin{equation}\label{eq:0 int}0=\int_{S^2}f^*\si=\int_\C\big(\pi_X\circ u_0\circ\phi\big)^*\si=\int_{B_1}(\pi_X\circ u_0)^*\si.
\end{equation}
Let $v_0\Colon \D\to E$ be the unique map such that $i_E\circ v_0=u_0$. Then $\pi_X\circ u_0=\pi\circ v_0$, and hence using Lemma \ref{le:pi si}, 
\[(\pi_X\circ u_0)^*\si=v_0^*\pi^*\si=v_0^*i_E^*\om=u_0^*\om.\]
Inserting this into (\ref{eq:0 int}), Claim \ref{claim:int D u 0} follows. \end{proof}
Claims \ref{claim:int D u u 0} and \ref{claim:int D u 0} imply that $\int_\D u^*\om=0$. It follows that $A(M,\om,N)=\infty$. This proves the second statement and completes the proof of Proposition \ref{prop:N}. 
\end{proof}
The next lemma was used in the proof of Theorem \ref{thm:leaf}.
\begin{lemma}\label{le:phi psi}Let $(M,\om)$ and $(M',\om')$ be symplectic manifolds, $\phi,\psi\in\Ham(M,\om)$, and $\phi',\psi'\in\Ham(M',\om')$. Then 
\[d^{M\x M',\om\oplus\om'}\big(\phi\x\phi',\psi\x\psi'\big)\leq d^{M,\om}(\phi,\psi)+d^{M',\om'}(\phi',\psi').\]
\end{lemma}
\begin{proof}[Proof of Lemma \ref{le:phi psi}]\setcounter{claim}{0} If $H\Colon [0,1]\x M\to \R$ and $H'\Colon [0,1]\x M'\to\R$ are smooth Hamiltonians generating $\psi^{-1}\circ\phi$ and ${\psi'}^{-1}\circ\phi'$ respectively, then the function
\[\wt H\Colon [0,1]\x \wt M\to \R,\quad \wt H(t,x,x'):=H(t,x)+H'(t,x')\]
generates $(\psi^{-1}\circ\phi)\x({\psi'}^{-1}\circ\phi')$. Furthermore, we have $\sup_{\wt M}\wt H(t,\cdot,\cdot)=\sup_MH(t,\cdot)+\sup_{M'}H'(t,\cdot),$ for every $t\in[0,1]$, and similarly for the infimum. It follows that $\Vert\wt H\Vert_{M\x M',\om\oplus\om'}=\Vert H\Vert_{M,\om}+\Vert H'\Vert_{M',\om'}$. The statement of Lemma \ref{le:phi psi} follows from this. 
\end{proof}
The next lemma was used in the proof of Theorem \ref{thm:L}. 
\begin{lemma}\label{le:H} Let $(M,\om)$ be a symplectic manifold, $\phi\in\Ham(M,\om)$, $K\sub M$ a compact subset, and $C>d(\phi,\id)$ a constant. Then there exists a function $H\in C^\infty_\c\big([0,1]\x M,\R\big)$ such that 
\begin{equation}\label{eq:H}\phi_{H}^1=\phi\textrm{ on }K,\quad \Vert H\Vert<C.
\end{equation}
\end{lemma}
\begin{proof}[Proof of Lemma \ref{le:H}]\setcounter{claim}{0} Let $M,\om,\phi,K$ and $C$ be as in the hypothesis. We choose a smooth function $\wt H\Colon [0,1]\x M\to\R$ that generates $\phi$ and satisfies $\Vert\wt H\Vert<C$. We also fix an open neighborhood $U\sub M$ of $K$ with compact closure, and we define $K':=\bigcup_{t\in[0,1]}\phi_{\wt H}^t(\bar U)$. We choose an open neighborhood $V\sub M$ of $K'$ such that $\bar V$ is compact. It follows from a $C^\infty$-version of Urysohn's Lemma for $\R^n$ (see for example Theorem 1.1.3 p.4 in \cite{KP}) and a partition of unit argument that there exists a smooth function $f\Colon M\to [0,1]$ such that $f^{-1}(0)=M\wo V$ and $f^{-1}(1)=K'$. We fix a point $x_0\in M$ and define 
\[H\Colon [0,1]\x M\to\R,\quad H(t,x):=f(x)\big(\wt H(t,x)-\wt H(t,x_0)\big).\]
Then the support of $H$ is contained in $\bar V$ and hence compact. Furthermore, for $t\in[0,1]$ and $x\in \bar U$ we have $H(t,\phi_{\wt H}^t(x))=\wt H(t,\phi_{\wt H}^t(x))-\wt H(t,x_0)$. It follows that $X_{H(t,\cdot)}(x)=X_{\wt H(t,\cdot)}(x)$, for $x\in K'$. This implies that $\phi_H^1(x)=\phi_{\wt H}^1(x)=\phi(x)$, for $x\in\bar U$, and therefore the first condition in (\ref{eq:H}) holds. Finally, observe that 
\[\max_{x\in M}H(t,x)\leq\sup_{x\in M}\wt H(t,x)-\wt H(t,x_0).\]
Combining this with a similar inequality for $\min_{x\in M}H(t,x)$, it follows that $\Vert H\Vert\leq\Vert\wt H\Vert$. Since $\Vert\wt H\Vert<C$, the second condition in (\ref{eq:H}) follows. This proves Lemma \ref{le:H}.
\end{proof}
The next lemma was used in the proof of Theorem \ref{thm:f}. Let $(M,\om)$ be a closed presymplectic manifold, and $H\sub TM$ an $\om$-horizontal distribution.  We denote by $\pi:(TM^{\om})^*\to M$ the canonical projection. We define $\iota_H$ and $\Om_{\om,H}$ as in (\ref{eq:iota H TM},\ref{eq:Om om H}). 
\begin{lemma}\label{le:f M} Let $f\in C^\infty(M,\R)$, $U\sub (TM^{\om})^*$ be an open neighborhood of the zero section on which $\Om_{\om,H}$ is non-degenerate, and $X$ the Hamiltonian vector field on $U$ generated by $f\circ \pi:M\to \R$ with respect to $\Om_{\om,H}$. Then $\pi_*X(x)\in H_{\pi(x)}$, for every $x\in U$. 
\end{lemma}
\begin{proof}[Proof of Lemma \ref{le:f M}]\setcounter{claim}{0} We fix $x:=(y,\al)\in U$, and denote by $\pr^H:T_yM\to T_yM^\om$ the canonical projection along $H$. We denote by $\pi':T^*M\to M$ the canonical projection, and by $i:(T_yM^\om)^*\to T_x(TM^\om)^*$ and $i':T^*_yM\to T_{\iota_H(x)}T^*M$ the canonical inclusions. 
\begin{claim}\label{claim:v al} For every $v\in T_xU$ and $\be\in (T_{y}M^{\om})^*$, we have 
\[\Om_{\om,H}(v,i\be)=\be\pr^H\pi_*v.\] 
\end{claim}
\begin{proof}[Proof of Claim \ref{claim:v al}] We have $\pi_*i\be=0$, and 
\[{\iota_H}_*i\be=\left.\frac{d}{dt}\right|_{t=0}\iota_H\big(y,\al+t\be\big)=\left.\frac{d}{dt}\right|_{t=0}\big(y,(\al+t\be)\pr^H\big)=i'(\be\pr^H).\]
It follows that 
\begin{eqnarray*}\Om_{\om,H}(v,i\be)&=&\om\big(\pi_*v,\pi_*i\be)+\om_\can\big({\iota_H}_*v,{\iota_H}_*i\be\big)\\
&=&0+\om_\can\big({\iota_H}_*v,i'(\be\pr^H)\big)\\
&=&\be\pr^H\pi'_*{\iota_H}_*v\\
&=&\be\pr^H\pi_*v.\end{eqnarray*}
This proves Claim \ref{claim:v al}.
\end{proof}
Claim \ref{claim:v al} implies that 
\[0=f_*\pi_*i\be=d(f\circ\pi)(x)i\be=\Om_{\om,H}(X(x),i\be)=\be\pr^H\pi_*X(x),\]
for every $\be\in (T_{y}M^{\om})^*$. It follows that $\pr^H\pi_*X(x)=0$, i.e. $\pi_*X(x)\in H_{y}$. This proves Lemma \ref{le:f M}.
\end{proof}
\subsection{Foliations}\label{subsec:fol}
In this subsection Proposition \ref{prop:hol} is proved. This result was used to define the linear holonomy of a foliation. The second result of this subsection is an estimate for the distance between the initial and end point of a path $x$ in a foliation, provided that these points lie in the same leaf, and $x$ is tangent to a given horizontal distribution. For the proof of Proposition \ref{prop:hol} we need the following lemma. Let $(M,\F)$ be a foliated manifold and $(U,\phi)\in\F$ a chart. We write $\phi=:(\phi^\xi,\phi^\eta)\Colon U\to\R^n=\R^{n-k}\x\R^k$. 
\begin{lemma}\label{le\Colon const} Let $F\sub M$ be a leaf of $\F$, $a\leq b$, and $u\Colon [a,b]\to F\cap U$ a continuous map. Then $\phi^\xi\circ u\Colon [a,b]\to \R^{n-k}$ is locally constant. 
\end{lemma}
\begin{proof}[Proof of Lemma \ref{le:const}]\setcounter{claim}{0} Let $(M,\F)$ be a foliated manifold. By definition, the leaf topology on $F$ is the topology $\tau^\F_F$ generated by the sets $\phi^{-1}(\{0\}\x\R^k)$, where $(U,\phi)\in\F$ is such that $\phi^{-1}(\{0\}\x\R^k)\sub F$. It is second countable, see for example Lemma 1.3. on p. 11 in the book \cite{Mol}. It follows that there exists a countable collection of surjective foliation charts $\phi_i:U_i\to\R^n$ $(i\in\N)$, such that $\big(\phi_i^{-1}(\{0\}\x\R^k)\big)_{i\in\N}$ is a basis for $\tau^\F_F$. Let $(U,\phi)\in\F$. Then $U\cap F\in\tau^\F_F$, and therefore there exists a subset $S\sub\N$ such that $U\cap F=\bigcup_{i\in S}U_i$. For each $i\in S$ compatibility of $\phi$ and $\phi_i$ implies that $\phi^\xi$ is constant on $U_i$. It follows that $\phi^\xi(U\cap F)\sub\R^{n-k}$ is at most countable. The statement of Lemma \ref{le:const} follows from this. 
\end{proof}
\begin{proof}[Proof of Proposition \ref{prop:hol}]\setcounter{claim}{0} Let $M,\F,F,a,b,x,N$ and $y_0$ be as in the hypothesis. To prove {\bf statement (\ref{prop:hol:u})}, let $T\Colon T_{y_0}N\to T_{x(a)}M$ be a linear map.
\begin{claim}\label{claim:f Y X} There exists a smooth map $f\Colon N\to M$ such that $f(y_0)=x(a)$ and $df(y_0)=T$. 
\end{claim}
\begin{proof}[Proof of Claim \ref{claim:f Y X}] We choose a quadruple $(U,V,\phi,\psi)$, where $U\sub M$ and $V\sub N$ are neighborhoods of $x(a)$ and $y_0$, respectively, and $\phi\Colon T_{x(a)}M\to U$ and $\psi\Colon V\to T_{y_0}N$ are diffeomorphisms, such that the following holds. Identifying $T_0(T_{x(a)}M)=T_{x(a)}M$ and $T_0(T_{y_0}N)=T_{y_0}N$, we have
\[\phi(0)=x(a),\quad d\phi(0)=\id_{T_{x(a)}M},\quad \psi(y_0)=0,\quad d\psi(y_0)=\id_{T_{y_0}N}.\] 
Furthermore, we choose a function $\rho\in C^\infty(T_{y_0}N, [0,1])$ such that $\rho=1$ in a neighborhood of $0$, and $\rho=0$ outside some compact subset of $T_{y_0}N$. We define $f(y):=\phi\circ(\rho\cdot T)\circ\psi(y)$ for $y\in V$, and $f(y):=y_0$, for $y\in N\wo V$. This map has the required properties. This proves Claim \ref{claim:f Y X}.
\end{proof}
We denote by $\pi_2\Colon [a,b]\x M\to M$ the projection onto the second factor.
\begin{claim}\label{claim:X} There exists a smooth section $s\Colon [a,b]\x M\to \pi_2^*T\F$ of compact support, such that $s(t,x(t))=\dot x(t)$, for every $t\in[a,b]$. 
\end{claim}
\begin{proof}[Proof of Claim \ref{claim:X}] For every $t\in[a,b]$ we choose a foliation chart $\phi_t:U_t\to\R^n$, such that $U_t\sub M$ is an open neighborhood of $x(t)$. Shrinking $U_t$ and reparametrizing $\phi_t$, we may assume that $\phi_t$ is surjective. We choose a finite subset $S\sub [a,b]$ such that $x([a,b])\sub U:=\bigcup_{t\in S}U_t$. We fix $t\in S$, and define 
\[A_t:=\big\{\big(t',\phi_t\circ x(t')\big)\,\big|\,x(t')\in U_t\big\}.\]
Since $\phi_t$ is surjective, $A_t$ is a closed subset of $[a,b]\x\R^n$. We choose a smooth extension $f_t\Colon [a,b]\x \R^n\to \R^k$ of the map $A_t\ni\big(t',\phi_t\circ x(t')\big)\mapsto (\phi^\eta_t)_*\dot x(t')\in\R^k$. We also fix a partition of unity $(\rho_t)_{t\in S}$ for $U$, subordinate to $(U_t)_{t\in S}$, and a smooth map $\rho\Colon U\to[0,1]$ with compact support, such that $\rho|_{x([a,b])}=1$. We define 
\[s\Colon [a,b]\x M\to \pi_2^*T\F,\, s(t',x'):=\rho(x')\sum_{t\in S}\rho_t(x')d\phi_t(x')^{-1}\big(0,f_t(t',\phi_t(x'))\big).\]
Here each summand on the right hand side is defined to be $0$ if $x'\not\in U_t$. The map $s$ has the required properties. This proves Claim \ref{claim:X}.
\end{proof}
We choose a map $f$ and a section $s$ as in Claims \ref{claim:f Y X} and \ref{claim:X}. Since $s$ has compact support, there exists a unique solution $u\Colon [a,b]\x N\to M$ of the equations
\[\d_tu(t,y)=s(t,y),\quad u(a,y)=f(y),\quad \forall t\in[a,b],\,y\in N.\]
This map has the required properties. This proves (\ref{prop:hol:u}). 

To prove {\bf statement (\ref{prop:hol:u u'})}, let $u$ and $u'$ be as in the hypothesis. Consider
\[S:=\big\{t\in[a,b]\,\big|\,\pr^\F d(u(t,\cdot))(y_0)=\pr^\F d(u'(t,\cdot))(y_0)\big\}.\]
By (\ref{eq:u u' a}) this set contains $a$. Furthermore, it is a closed subset of $[a,b]$. 
\begin{claim}\label{claim:S} $S$ is open. 
\end{claim}
\begin{proof}[Proof of Claim \ref{claim:S}] Let $t_0\in S$. We choose a chart $(U,\phi)\in\F$ such that $x(t_0)\in U$, and a number $\eps>0$ such that $x\big([t_0-\eps,t_0+\eps]\cap[a,b]\big)\sub U$. We define
\[V:=\big\{y\in N\,\big|\,u(t,y),u'(t,y)\in U,\,\forall t\in[t_0-\eps,t_0+\eps]\cap[a,b]\big\}.\]
This is an open subset of $N$. Furthermore, by the first condition in (\ref{eq:u 0 x}) we have $y_0\in V$. Let $x_0\in U$. Then by definition, the map $d\phi^\xi(x_0)\Colon T_{x_0}M\to\R^{n-k}$ is surjective and has kernel $T_{x_0}\F$. It follows there exists a unique linear isomorphism $\Phi_{x_0}\Colon N_{x_0}\F=T_{x_0}M/T_{x_0}\F\to \R^{n-k}$ satisfying $\Phi_{x_0}\pr^\F_{x_0}=d\phi^\xi(x_0)$. We fix $t\in [t_0-\eps,t_0+\eps]$. By Lemma \ref{le:const}, we have
\[\phi^\xi\circ u(t,y)=\phi^\xi\circ u(t_0,y),\quad \phi^\xi\circ u'(t,y)=\phi^\xi\circ u'(t_0,y),\]
for every $y\in V$. It follows that on $TV$, we have
\begin{eqnarray}\nn&d\phi^\xi(x(t))d(u(t,\cdot))=d\phi^\xi(x(t_0))d(u(t_0,\cdot))=\Phi_{x(t_0)}\pr^\F_{x(t_0)}d(u(t_0,\cdot)),&\\
\nn &d\phi^\xi(x(t))d(u'(t,\cdot))=\Phi_{x(t_0)}\pr^\F_{x(t_0)}d(u'(t_0,\cdot)).&\end{eqnarray}
Since $t_0\in S$, this implies that $d\phi^\xi(x(t))d(u(t,\cdot))(y_0)=d\phi^\xi(x(t))d(u'(t,\cdot))(y_0)$. Using the equality $\Phi_{x(t)}^{-1}d\phi^\xi(x(t))=\pr^\F_{x(t)}$, it follows that $\pr^\F d(u(t,\cdot))(y_0)=\pr^\F d(u'(t,\cdot))(y_0)$. Hence $S$ is open. This proves Claim \ref{claim:S}.
\end{proof}
Using Claim \ref{claim:S}, it follows that $S=[a,b]$. This proves (\ref{prop:hol:u u'}) and completes the proof of Proposition \ref{prop:hol}.
\end{proof}
The next result was used in the proof of Theorem \ref{thm:f}. %
\begin{prop}\label{prop:X F} Let $M$ be a closed manifold, $\F$ a regular foliation on $M$, $H\sub TM$ an $\F$-horizontal distribution, and $g$ a Riemannian metric on $M$. Then there exists a constant $C$ such that for every $t\geq0$ and every $x\in C^\infty([0,t],M)$ the following holds. If $\dot x(s)\in H_{x(s)}$, for every $s\in[0,t]$, and $x(t)\in\F_{x(0)}$, then $d(x(0),x(t))\leq C\ell(x)^2$. Here $\ell$ and $d$ denote the length and distance functions with respect to $g$ respectively.
\end{prop}
\begin{proof}[Proof of Proposition \ref{prop:X F}]\setcounter{claim}{0} We denote by $n$ and $k$ the dimension of $M$ and of the leaves of $\F$ respectively. 
\begin{claim}\label{claim:A} There exists a finite atlas $\A$ of surjective foliation charts $\phi:U\to\R^n$, such that $\bigcup_{(U,\phi)\in\A}\phi^{-1}(B_1)=M$, and for every $(U,\phi)\in\A$ and $x\in U$ the set $\phi(\F_x\cap U)$ is connected.
\end{claim}
\begin{proof}[Proof of Claim \ref{claim:A}] Since by assumption $\F$ is regular it follows from Lemma \ref{le:X/R}(\ref{le:X/R:fol}) below that there exists a smooth structure on the set $M'$ of leaves of $\F$, such that the canonical projection $\pi:M\to M'$ is a submersion. Since $M$ is closed, a result by Ehresmann implies that $\pi$ is a fiber bundle. (See the proposition on p.~31 in \cite{Eh}.) Let $x\in M$. We choose a local trivialization $\psi_x:U'_x\x \F_x\to M$ of $\pi$, where $U'_x$ is an open subset of $M'$, such that $x\in \psi_x(U'_x\x\F_x)$. By combining $\psi_x$ with charts of $M'$ and $\F_x$ containing the points $\pi(x)$ and $x$, respectively, we obtain a foliation chart $(U_x,\phi_x)$ for $M$ such that $x\in U_x$ and $\phi^i_x(y)=\phi_x^i(x)$, for every $y\in U_x\cap \F_x$, $i=n-k+1,\ldots,n$. By shrinking the domain and target of $\phi_x$ and rescaling, we may assume \Wlog that $\phi_x(U_x)=\R^n$. By compactness of $M$ there exists a finite subset $S\sub M$ such that $\bigcup_{x\in S}\phi_x^{-1}(B_1)=M$. The set $\A:=\big\{(U_x,\phi_x)\,|\,x\in S\}$ has the required properties. This proves Claim \ref{claim:A}.
\end{proof}
We choose an atlas $\A$ as in Claim \ref{claim:A}. For $(U,\phi)\in\A$ we define $\eps_\phi$ to be the distance (with respect to $g$) between $\phi^{-1}(\bar B_1)$ and $M\wo \phi^{-1}(B_2)$, and we set $\eps:=\min\{\eps_\phi\,|\,(U,\phi)\in\A\}$. Let $(U,\phi)\in\A$. We define the map $\al_\phi:\R^n\to\R^{k\x (n-k)}$ as follows. Namely, for $x\in \R^n$ we define $\al_\phi(x)$ to be the unique real $k\x (n-k)$ matrix satisfying $\big\{(w,\al_\phi(x)w)\,|\,w\in\R^{n-k}\big\}=\phi_*H_{\phi^{-1}(x)}$. Since $H$ is horizontal, $\al$ is well-defined. We denote by $|\cdot|_0$ the standard norm on Euclidian space, by $|v|$ the norm of vector $v\in TM$ with respect to $g$, and by $\pi_1:\R^n=\R^{n-k}\x\R^k\to\R^{n-k}$ the canonical projection onto the first component. We choose a constant $C$ such that 
\begin{eqnarray}\label{eq:v 0}&|v|_0\leq C\big|d(\phi^{-1})(x)v\big|,&\\
\label{eq:x y} &d\big(\phi^{-1}(x),\phi^{-1}(y)\big)\leq C|x-y|_0,&\\\label{eq:d al}&\big|(d\al_\phi(x)v)w\big|_0\leq C|v|_0\,|w|_0,&\end{eqnarray} 
for every $(U,\phi)\in\A$, $x,y\in\bar B_2$, $v\in\R^n$, and $w\in\R^{n-k}$. Let $t_0\in[0,\infty)$, and $x\in C^\infty([0,t_0],M)$ a path such that $\dot x(t)\in H_{x(t)}$, for every $t\in[0,t_0]$, and $x(t_0)\in\F_{x(0)}$.
\begin{claim}\label{claim:d C} If $\ell(x)\leq \eps$ then $d(x(0),x(t_0))\leq 2C^3\ell(x)^2$. 
\end{claim}
\begin{proof}[Proof of Claim \ref{claim:d C}] Assume that $\ell(x)\leq \eps$. We choose a chart $(U,\phi)\in\A$ such that $x(0)\in\phi^{-1}(B_1)$. By the choice of $\eps$ it follows that $x(t)\in\phi^{-1}(B_2)$, for every $t\in[0,t_0]$. Hence we may define $(a,b):=\phi\circ x:[0,t_0]\to\R^n=\R^{n-k}\x\R^k$. By the choice of $\A$, the set $\phi(\F_{x(0)}\cap U)$ is connected, hence it equals $\{a(0)\}\x\R^k$. Since $x(t_0)\in\F_{x(0)}$, it follows that $a(t_0)=a(0)$. Therefore, 
\begin{equation}\label{eq:a b t 0}\big|(a,b)(t_0)-(a,b)(0)\big|_0=|b(t_0)-b(0)|_0.\end{equation}
By the definition of $\al_\phi$ and the hypothesis $\dot x(t)\in H_{x(t)}$, we have
\begin{equation}\label{eq:b t 0}b(t_0)-b(0)=\int_0^{t_0}\dot b(t)dt=\int_0^{t_0}\al_\phi\circ (a,b)(t)\dot a(t)dt.
\end{equation}
We define $(u,v):[0,1]\x[0,t_0]\to\R^n=\R^{n-k}\x\R^k$ by
\[(u,v)(s,t):=(a,b)(0)+s\big((a,b)(t)-(a,b)(0)\big).\]
Then $(u,v)(1,t)=(a,b)(t)$ and $\d_tu(0,t)=0$, for every $t\in[0,t_0]$, and therefore
\begin{eqnarray}\nn&&\int_0^{t_0}\al_\phi\circ (a,b)(t)\dot a(t)dt\\
\nn&&=\int_0^{t_0}\int_0^1\d_s\big(\al_\phi\circ(u,v)\d_tu\big)\,ds\,dt\\
\nn&&=\int_0^{t_0}\int_0^1\Big(\d_s\big(\al_\phi\circ(u,v)\big)\d_tu-\d_t\big(\al_\phi\circ(u,v)\big)\d_su\\
\nn&&\phantom{\int_0^{t_0}\int_0^1}+\d_t\big(\al_\phi\circ(u,v)\d_su\big)\Big)\,ds\,dt\\
&&=\label{eq:int int}\int_0^{t_0}\int_0^1\Big(\big(d\al_\phi\,\d_s(u,v)\big)\d_tu-\big(d\al_\phi\,\d_t(u,v)\big)\d_su\Big)\,ds\,dt\\
\nn&&\phantom{=}+\int_0^1\al_\phi\circ(u,v)\d_su\,ds\Big|_{t=0}^{t_0}.
\end{eqnarray}
Since $a(t_0)=a(0)$, we have $\d_su(s,t)=0$, for $s\in[0,1]$ and $t=0,t_0$. Therefore, the last term in (\ref{eq:int int}) vanishes. Using (\ref{eq:b t 0}) and (\ref{eq:d al}), it follows that 
\begin{eqnarray}\nn|b(t_0)-b(0)|_0&\leq&C\int_0^{t_0}\int_0^1\big(|\d_s(u,v)|_0|\d_tu|_0+|\d_t(u,v)|_0|\d_su|_0\big)\,ds\,dt\\
\nn&\leq&2C\int_0^{t_0}\int_0^1|\d_s(u,v)|_0|\d_t(u,v)|_0\,ds\,dt\\
\nn&\leq& 2C\int_0^{t_0}\big|(a,b)(t)-(a,b)(0)\big|_0\left|\frac d{dt}(a,b)(t)\right|_0\,dt\\
\nn&\leq&2C^2\max\big\{\big|(a,b)(t)-(a,b)(0)\big|_0\,\big|\,t\in[0,t_0]\big\}\ell(x)\\
\nn&\leq&2C^2\int_0^{t_0}\left|\frac d{dt}(a,b)(t)\right|_0dt\,\ell(x)\\
\label{eq:2C3}&\leq&2C^3\ell(x)^2.
\end{eqnarray}
Here in the fourth and the last step we used (\ref{eq:v 0}). Combining (\ref{eq:x y},\ref{eq:a b t 0},\ref{eq:2C3}), we obtain
\[d(x(t_0),x(0))\leq C\big|(a,b)(t_0)-(a,b)(0)\big|_0\leq 2C^3\ell(x)^2.\]
This proves Claim \ref{claim:d C}. 
\end{proof}
Note that in the case $\ell(x)>\eps$ we have $d(x(0),x(t_0))\leq\eps^{-1}\ell(x)^2$. Combining this with Claim \ref{claim:d C}, we obtain $d(x(0),x(t_0))\leq \max\big\{2C^3,\eps^{-1}\big\}\ell(x)^2$. This concludes the proof of Proposition \ref{prop:X F}.
\end{proof}

\subsection{Further auxiliary results}\label{subsec:further}
\begin{lemma}[Smooth structures on quotients]\label{le:X/R} Let $M$ be a set with a smooth structure and $R$ an equivalence relation on $M$. Then the following holds.
\begin{enua}\item\label{le:X/R:subm} There is at most one smooth structure on $M':=M/R$ such that the quotient map $\pi\Colon M\to M'$ is a submersion.\\

\noindent \emph{Assume now that $R$ is the leaf relation of some foliation $\F$ on $M$. Then:}
\item\label{le:X/R:ker} If there is a smooth structure on $M'$ as in (\ref{le:X/R:subm}) then $\ker d\pi(x)=T_x\F$, for every $x\in M$.
\item\label{le:X/R:Hausdorff} $R$ is a closed subset of $M\x M$ if and only if $M'$ (equipped with the quotient topology) is Hausdorff.\\

\noindent \emph{Assume that $R$ is the leaf relation of some foliation $\F$ and the induced topology on $M$ is Hausdorff and second countable. Then:}
\item\label{le:X/R:second} $M'$ is second countable.
\item \label{le:X/R:fol} The following conditions are equivalent.
\begin{enui}
\item\label{le:X/R:smooth} There exists a smooth structure on $M'$ as in (\ref{le:X/R:subm}).
\item\label{le:X/R:sub} $R$ is a submanifold of $M\x M$. 
\end{enui}
\item\label{le:X/R:hol} Assume that there is a smooth structure $\A$ on $M'$ as in (\ref{le:X/R:subm}). Let $F$ be a leaf of $\F$, $x\in C^\infty([0,1],F)$, and $v_i\in T_{x(i)}M$, for $i=0,1$, be such that $\pr^\F v_1=\hol^\F_x\pr^\F v_0$. Then $d\pi(x(0))v_0=d\pi(x(1))v_1,$ where the differentials are defined with respect to $\A$. 
\end{enua}
\end{lemma}
\begin{rmk}\label{rmk:X} Let $M$ be a set with a smooth structure and $R$ an equivalence relation on $M$. We denote now by $\pi_1\Colon R\to M$ the projection onto the first factor. Then by a theorem by Godement, condition (i) of part (\ref{le:X/R:fol}) above holds if and only if (ii) is satisfied and $\pi_1$ is a submersion. (See for example Theorem 3.5.25 in the book \cite{AMR}.) \end{rmk}
\begin{proof}[Proof of Lemma \ref{le:X/R}]\setcounter{claim}{0} Let $M$ be a set with a smooth structure, and $R$ an equivalence relation on $M$. {\bf Statement (\ref{le:X/R:subm})} follows from Proposition 3.5.21(iii) in the book \cite{AMR}. 

Assume now that $R$ is the leaf relation of some foliation $\F$ on $M$. 

In order to prove {\bf (\ref{le:X/R:ker})}, let $F\in M'$ be a leaf. Since $F$ is a regular value of $\pi$, the Implicit Function Theorem implies that $\pi^{-1}(F)=F\sub M$ is a submanifold, and $T_xF=\ker d\pi(x)$, for every $x\in F$. On the other hand, it follows from the definitions that $T_xF=T_x\F$. This proves (\ref{le:X/R:ker}).

To see {\bf(\ref{le:X/R:Hausdorff})}, observe that the map $\pi$ is open. This follows for example from the corollary on p.~19 in \cite{Mol}. (The proof goes through if $M$ is not Hausdorff or second countable.) Therefore, (\ref{le:X/R:Hausdorff}) follows from an elementary argument, see for example Lemma 2.3 p. 60 in \cite{Bo}. 

Assume now also that the topology on $M$ is Hausdorff and second countable. Using openness of $\pi$, {\bf statement (\ref{le:X/R:second})} follows from Lemma 2.4 p. 60 in \cite{Bo}. Furthermore, by Remark \ref{rmk:X}, {\bf(\ref{le:X/R:fol})} is a consequence of the following.
\begin{claim}\label{claim:pr 1} The projection $\pi_1$ is a submersion. 
\end{claim}
\begin{pf}[Proof of Claim \ref{claim:pr 1}] Let $(x_0,x_1)\in R$. We choose a path $x\in C^\infty([0,1],\F_{x_0})$ such that $x(i)=x_i$, for $i=0,1$. We set $a:=0$, $b:=1$, $N:=M$, $y_0:=x_0$, and $T:=\id_{T_{x_0}M}$. Applying Proposition \ref{prop:hol}(\ref{prop:hol:u}) there exists a map $u\in C^\infty([0,1]\x M,M)$ such that the conditions (\ref{eq:u 0 x}) and (\ref{eq:d u a y}) hold. By (\ref{eq:u 0 x}) the map $f\Colon M\to M\x M$ defined by $f(y):=\big(u(0,y),u(1,y)\big)$ takes values in $R$ and satisfies $f(x_0)=(x_0,x_1)$. Equality (\ref{eq:d u a y}) implies that 
\[d\pi_1(x_0,x_1)df(x_0)=d(\pi_1\circ f)(x_0)=d(u(0,\cdot))(x_0)=\id_{T_{x_0}M}.\]
It follows that $d\pi_1(x_0,x_1)$ is surjective, hence $\pi_1$ is a submersion. This proves Claim \ref{claim:pr 1}.
\end{pf}

We show {\bf(\ref{le:X/R:hol}).} We choose a map $u\in C^\infty\big([a,b]\x N_{x(a)}\F,M\big)$ as in the definition (\ref{eq:hol F x}) of $\hol^\F_x$. Then for every $s\in\R$, we have
\[\pi\circ u(1,\pr^\F sv_0)=\pi\circ u(0,\pr^\F sv_0).\]
Differentiating this identity with respect to $s$, we obtain
\begin{equation}\label{eq:d pi x 1}\pi_*u(1,\cdot)_*\pr^\F v_0=\pi_*u(0,\cdot)_*\pr^\F v_0.\end{equation}
On the other hand, the equality $\pr^\F_{x(0)}d(u(0,\cdot))(0)=\id_{N_{x(0)}\F}$ implies that $v_0-u(0,\cdot)_*\pr^\F v_0\in T_{x(0)}\F$. Using statement (\ref{le:X/R:ker}), it follows that 
\begin{equation}\label{eq:pi u 0}\pi_*u(0,\cdot)_*\pr^\F v_0=\pi_*v_0.\end{equation}
By assumption we have
\[\pr^\F v_1=\hol_x^\F\pr^\F v_0=\pr^\F u(1,\cdot)_*\pr^\F v_0.\]
Using again statement (\ref{le:X/R:ker}), it follows that $\pi_*v_1=\pi_*u(1,\cdot)_*\pr^\F v_0$. Combining this with (\ref{eq:d pi x 1}) and (\ref{eq:pi u 0}), we obtain $\pi_*v_0=\pi_*v_1,$ as claimed. This proves (\ref{le:X/R:hol}) and concludes the proof of Lemma \ref{le:X/R}. 
\end{proof}
The next lemma was used in the proof of Theorem \ref{thm:f}. Let $M$ be a $C^1$-manifold, and $X$ a complete $C^1$-vector field on $M$. If $g$ is a Riemannian metric on $M$ then we denote by $\ell$ and $d$ the induced length funcional and distance function, respectively. Furthermore, for a pair $(t,x_0)\in[0,\infty)\x M$ we write $\ell(t,x_0):=\ell\big([0,t]\ni s\mapsto\phi_X^s(x_0)\in M\big)$.
\begin{lemma}[Fast almost periodic orbits]\label{le:ell f} Let $(M,g)$ be a Riemannian $C^2$-manifold, $X$ a $C^1$-vector field on $M$ with compact support, and $f:[0,\infty)\to[0,\infty)$ a continuous function such that $f(0)=0$. Then there exists a constant $\eps>0$ such that for every $(t,x_0)\in[0,\eps]\x M$ satisfying $d\big(x_0,\phi_X^t(x_0)\big)\leq \ell(t,x_0)f(\ell(t,x_0))$, we have $X(x_0)=0$. 
\end{lemma}
The proof of this lemma is based on an idea from the proof of Proposition 17, p. 184 in the book \cite{HZ}. We need the following.
\begin{rmk}\label{rmk:t x} If $t\geq0$, and $x\in W^{1,1}\big([0,t],\R\big)$ is such that $\int_0^tx(s)ds=0$ then 
\begin{equation}\label{eq:x x dot}\Vert x\Vert_{L^1([0,t])}\leq t\Vert\dot x\Vert_{L^1([0,t])}.\end{equation}
To see this, note that $\int_0^tx(s)ds=0$ implies that there is a point $t_0\in[0,t]$ such that $x(t_0)=0$. It follows that for every $s\in[0,t]$, we have $|x(s)|=\big|\int_{t_0}^s\dot x(s)ds\big|\leq\int_0^t |\dot x(s)|ds$. Inequality (\ref{eq:x x dot}) is a consequence of this.
\end{rmk}
\begin{proof}[Proof of Lemma \ref{le:ell f}]\setcounter{claim}{0} Let $M,g,X$ and $f$ be as in the hypothesis. We denote by $K\sub M$ the support of $X$, by $n$ the dimension of $M$, for a vector $v\in TM$ we denote by $|v|$ its norm with respect to $g$, and for $v\in\R^n$ we define $|v|_1:=\sum_{i=1}^n|v^i|$. We choose a finite set $\A$ of surjective $C^2$-charts $\phi:U\sub M\to\R^n$, such that $K\sub\bigcup_{(U,\phi)\in\A}\phi^{-1}(B_1)$. Furthermore, we choose a constant $C$ such that 
\begin{equation}\label{eq:|x-y|}|x-y|_1\leq Cd\big(\phi^{-1}(x),\phi^{-1}(y)\big),\quad\big|d(\phi^{-1})(x)v\big|\leq C|v|_1,
\end{equation}
for every $(U,\phi)\in\A$, $x,y\in \bar B_2$, and $v\in \R^n$. For $(U,\phi)\in\A$ we define $\eps_\phi$ to be the distance between $\phi^{-1}(\bar B_1)$ and $M\wo \phi^{-1}(B_2)$, and we define 
\begin{equation}\label{eq:eps}\eps_1:=\min\left\{\frac{\eps_\phi}{\max_K|X|}\,\Big|\,(U,\phi)\in\A\right\}. 
\end{equation}
For a linear map $T:\R^n\to\R^n$ we denote 
\[|T|_\op:=\max\big\{|Tv|_1\,\big|\,v\in\R^n,\,|v|_1=1\big\}.\]
Since by assumption $f$ is continuous and $f(0)=0$, there exists $\eps_2>0$ such that for every $(U,\phi)\in\A$, and every $a\in\big[0,\eps_2\max_K|X|\big]$, we have
\begin{equation}\label{eq:eps Vert}\eps_2\max_{\bar B_2}|d(\phi_*X)|_\op+C^2f(a)<1.
\end{equation}
We define $\eps:=\min\{\eps_1,\eps_2\}$. Let $(t,x_0)\in[0,\eps]\x M$ be such that $d\big(x_0,\phi_X^t(x_0)\big)\leq \ell(t,x_0)f(\ell(t,x_0))$. We define $x:[0,t]\to M$ by $x(s):=\phi_X^s(x_0)$. We choose a chart $(U,\phi)\in\A$ such that $x_0\in\phi^{-1}(B_1)$. For $s\in[0,t]$, we have 
\[d(x_0,x(s))\leq\ell(x)=\int_0^s|X\circ x(s')|\,ds'\leq \eps_1\max_K|X|\leq\eps_\phi.\]
It follows that $x([0,t])\sub \phi^{-1}(\bar B_2)$. Hence we may define 
\[y:=\phi\circ x,\quad v:=y(0)-y(t),\quad f:[0,1]\to \R^n,\,f(s):=\dot y(s)+v/t.\]
The equality $\dot x=X\circ x$ implies that $f$ is $C^1$, and that 
\begin{equation}\label{eq:dot f}\dot f=\ddot y=\big((\phi_*X)\circ y\big)\Dot=d(\phi_*X)(y)\dot y.\end{equation}
Furthermore, $\int_0^tf(s)ds=0$. Therefore, denoting by $\Vert f\Vert_1$ the $L^1$-norm of $f$ with respect to $|\cdot|_1$, Remark \ref{rmk:t x} with $x$ replaced by each component of $f$ implies that
\begin{equation}\label{eq:Vert f}\Vert f\Vert_1\leq t\Vert\dot f\Vert_1\leq t\max_{\bar B_2}|d(\phi_*X)|_\op\Vert \dot y\Vert_1.
\end{equation}
Here in the second inequality we used (\ref{eq:dot f}). By (\ref{eq:|x-y|}) and one of the hypotheses of the lemma, we have
\begin{equation}\label{eq:v 1}|v|_1\leq Cd(x_0,x(t))\leq C\ell(t,x_0)f\big(\ell(t,x_0)\big)\leq C^2\Vert\dot y\Vert_1f\big(\ell(t,x_0)\big).\end{equation}
Since $\dot y=f-v/t$, inequalities (\ref{eq:Vert f},\ref{eq:v 1}) yield
\begin{equation}\label{eq:Vert dot y}\Vert\dot y\Vert_1\leq\Vert f\Vert_1+|v|_1\leq\Big(\eps_2\max_{\bar B_2}|d(\phi_*X)|_\op+C^2f\big(\ell(t,x_0)\big)\Big)\Vert\dot y\Vert_1.\end{equation}
Since $\ell(t,x_0)\leq t\max_K|X|\leq\eps_2\max_K|X|$, inequality (\ref{eq:eps Vert}) holds with $a:=\ell(t,x_0)$. Combining this with (\ref{eq:Vert dot y}), it follows that $\Vert\dot y\Vert_1=0$. Hence $y$ is constant, and the same holds for $x=\phi^{-1}\circ y$. This proves Lemma \ref{le:ell f}. 
\end{proof}
The next lemma implies that the Hofer semi-norm given by (\ref{eq:Vert H}) is well-defined. 
\begin{lemma}\label{le:f} Let $X$ be a topological space and $f\Colon [0,1]\x X\to \R$ a continuous function. Assume that there exists a sequence of compact subsets $K_\nu\sub X$, $\nu\in\N$ such that $\bigcup_\nu K_\nu=X$. Then the map 
\[[0,1]\ni t\mapsto \sup_{x\in X}f(t,x)\]
is Borel measurable. 
\end{lemma}
\begin{proof}[Proof of Lemma \ref{le:f}]\setcounter{claim}{0} We choose a sequence $K_\nu\sub X$, $\nu\in\N$, as in the hypothesis, and we define 
\[f_\nu\Colon [0,1]\to \R,\quad f_\nu(t):=\max\big\{f(t,x)\,\big|\,x\in K_\nu\big\}.\]
Then $f_\nu$ is continuous, for every $\nu$, and $f(t)=\sup_{\nu\in\N} f_\nu(t)$, for every $t\in[0,1]$. Hence an elementary argument implies that $f$ is Borel measurable. This proves Lemma \ref{le:f}.
\end{proof}

\end{document}